\documentclass{article}
\usepackage{fullpage}
\usepackage{amssymb,dsfont,color}
\usepackage{amsthm}
\usepackage{amsmath}
\usepackage{mathrsfs}
\usepackage{enumerate}

\newtheorem{thm}{Theorem}

\newtheorem{remark}[thm]{Remark}
\newtheorem{lem}[thm]{Lemma}
\newtheorem{cor}[thm]{Corollary}
\newtheorem{definition}[thm]{Definition}
\newtheorem{example}[thm]{Example}
\def\ds{\displaystyle}
\def\s{\sigma}
\def\R{\mathbb{R}}
\def\S{\mathcal{S}}
\def\C{\mathcal{C}}
\def\Q{\mathcal{Q}}
\def\s{\sigma}

\definecolor{orange}{rgb}{0.99,0.69,0.07}

\begin{document}

\title{Lyapunov characterization of uniform exponential stability for nonlinear infinite-dimensional systems}

\author{Ihab Haidar\thanks{Quartz EA 7393, ENSEA, Cergy-Pontoise, France, {\tt ihab.haidar@ensea.fr}}, Yacine Chitour\thanks{Laboratoire des Signaux et Syst\`emes (L2S),  Universit\'e Paris-Saclay, CNRS, CentraleSup\'elec,  Universit\'e Paris-Saclay, Gif-sur-Yvette, France, {\tt yacine.chitour@l2s.centralesupelec.fr}}, Paolo Mason\thanks{CNRS \& Laboratoire des Signaux et Syst\`emes (L2S),  Universit\'e Paris-Saclay, CNRS, CentraleSup\'elec, Gif-sur-Yvette, France, {\tt paolo.mason@l2s.centralesupelec.fr}}, and Mario Sigalotti\thanks{Laboratoire Jacques-Louis Lions (LJLL), Inria, Sorbonne Universit\'e, Universit\'e de Paris, CNRS, Paris, France, {\tt mario.sigalotti@inria.fr}}}

\maketitle

\begin{abstract}
In this paper we deal with infinite-dimensional nonlinear forward complete dynamical systems which are subject to external disturbances. We first extend the well-known Datko lemma to the framework of the considered class of systems. Thanks to this generalization, we provide 
characterizations of the uniform (with respect to disturbances) local, semi-global, and global exponential stability, through the existence of coercive and non-coercive Lyapunov functionals. The importance of the obtained results is underlined through some applications concerning 1) exponential stability of nonlinear retarded systems with piecewise constant delays, 2) exponential stability preservation under sampling for semilinear control switching systems, and 3) the link between input-to-state stability and exponential stability of semilinear switching systems. 
\end{abstract}


%%%%%%%%%%%%%%%%%%%%%%%%%%%%%%%%%%%%%%%%%%%%%%%%%%%%%%%%%%%%%%%%%%%%%%%%%%%%%%%%

\section{Introduction}
Various works have been recently devoted to the characterization of the stability of infinite-dimensional systems
in Banach spaces  through {\it non-coercive} and {\it coercive} Lyapunov functionals (see, e.g.,~\cite{Haidar-Automatica, Hante2011, 7402277, Mironchenko2019, MIRONCHENKO2018}). 
By non-coercive Lyapunov functional, we mean a positive definite functional decaying along the 
trajectories of the system which satisfies 
\begin{equation*}\label{weak}
0<V(x)\leq \alpha(\|x\|), \quad \forall~x\in X\backslash\{0\},
\end{equation*} 
where $X$ is the ambient Banach space and $\alpha$ belongs to the class ${\cal{K}_{\infty}}$ of 
continuous increasing 
bijections from $\mathbb{R}_+$ to $\mathbb{R}_+$.  Such a function $V$ would be coercive if there existed
$\alpha_0\in{\cal{K}_{\infty}}$ such that $V(x)\geq \alpha_0(\|x\|)$ for every $x\in X$. In~\cite{MIRONCHENKO2018} it has been proved that the existence of a coercive Lyapunov functional $V$ 
represents a necessary and sufficient condition for the global asymptotic stability for a general class of infinite-dimensional forward complete dynamical systems. On the 
other hand, the existence of 
%contrary, 
a non-coercive Lyapunov functional does not guarantee
%is not sufficient in itself in order to get the 
global asymptotic stability and some additional regularity assumption on the dynamics is needed (see, e.g., \cite{Hante2011, MIRONCHENKO2018}). Converse Lyapunov theorems can be helpful for many applications, such as stability analysis of interconnected systems~\cite{HaidarChapter2019} and for the characterization of  input-to-state stability %characterization \ih{problem} 
(see, e.g., \cite{8618712,PEPE2017,SONTAG1995351}). 
%Results dealing with non-coercive Lyapunov functionals are better suited to derive stability results, 
%while those dealing with coercive Lyapunov functionals provide more information on a nonlinear system that is known to be  uniformly %globally 
%asymptotically stable.
%\footnote{\paolo{maybe replace this sentence with something like ``
Stability results %dealing with 
based on
non-coercive Lyapunov functionals may be more easily applied in practice, while the existence of a coercive Lyapunov functional may be exploited to infer
%reveal/uncover/identify 
additional information on a stable nonlinear system.
%''} \yacine{I follow that}   }

Here, we consider the same class of abstract forward complete dynamical systems, % which are
 subject to a shift-invariant set of disturbances, as in~\cite{MIRONCHENKO2018}. The novelty of our approach is that we focus on exponential (instead of asymptotic) stability. For the rest of the paper, the word \textit{uniform} will refer to uniformity with respect to disturbances.
We provide theorems characterizing different types of uniform local, semi-global, and global exponential stability, through the existence of non-coercive and  coercive Lyapunov functionals. Using a standard converse Lyapunov approach, we prove that uniform semi-global exponential stability is characterized by the existence of a 1-parameter family of Lyapunov functionals, each of them decaying uniformly on a bounded set, while the union of all such bounded sets is equal to the entire Banach space $X$. 
%Lyapunov function whose estimated size and rate of decay depend on the norm of the state.
Concerning the non-coercive case, we first give a generalization of the Datko lemma~\cite{Datko, Pazy1972}. Recall that the latter characterizes the exponential behavior of a linear $C_0$-semigroup in a Banach space 
in terms of a uniform estimate
%through the convergence 
of the $L^p$-norm of the solutions. This result has been extended in~\cite{Ichikawa} to the framework of nonlinear semigroups. Here, we generalize the Datko lemma to the considered class of infinite-dimensional forward complete dynamical systems. 
Thanks to such a generalization, we prove that the existence of a non-coercive Lyapunov functional is sufficient, under a uniform growth estimate on the solutions of the system,
for the uniform 
%local, semi-global, and global 
exponential stability. The importance of the obtained results is underlined through some applications
as described in the sequel.

Retarded functional differential equations form an interesting class of infinite-dimensional systems that we cover by our approach. Converse Lyapunov theorems have been developed for systems described by retarded and neutral functional differential equations (see, e.g., \cite{Karafyllis2006, Pepe-IJC2013}). Such results have been recently extended in~\cite{Haidar-Automatica} to switching linear retarded systems through coercive and non-coercive Lyapunov characterizations. After representing a nonlinear retarded functional differential equation as an abstract forward complete dynamical system, all the 
%provided theorems 
characterizations of uniform exponential stability provided in the first part of the paper
can be applied to this particular class of infinite-dimensional systems. In particular, we characterize the uniform global exponential stability of a retarded %(piecewise constant delay) 
functional differential 
equation
%\ih{equations} 
in terms of the existence of a non-coercive Lyapunov functional.

Another interesting problem when dealing with a continuous-time model is the practical implementation of a designed  feedback control. Indeed, in practice, due to numerical and technological limitations (sensors, actuators, and digital interfaces),
%softwares), 
a continuous measurement of the output and a continuous implementation of a feedback control are impossible. This means that the implemented input is, for almost every time, different from the designed controller. Several methods have been developed in the literature of ordinary differential equations for sampled-data observer design under discrete-time measurements (see, e.g., \cite{5718046, 5208358, LUCIEN20172941, MAM2015}), and for sampled-data control design guaranteeing a globally stable closed-loop system (see, e.g., \cite{Ackerman, HETEL2017309}). 
Apart from time-delays systems (see, e.g.,~\cite{Fridman2010, Karafyllis2012, PEPE2017295} for sampled-data control and \cite{MAM2015, MAZENC201674} for sampled-data observer design), few results exist for infinite-dimensional systems.  
The difficulties come from the fact that the developed methods do not directly apply to the infinite-dimensional case, for which even the well-posedness of sampled-data control dynamics is not obvious~(see, e.g., \cite{Karafyllis2018} for more details). Some interesting results have been obtained for infinite-dimensional linear systems~\cite{Karafyllis2018, Logemann2003, Wakaiki2019}. In the nonlinear case no standard methods have been developed and the problem is treated case by case~\cite{Koga2019}. Here, we focus on the particular problem of feedback stabilization under sampled output measurements of an abstract semilinear infinite-dimensional system. In particular, we consider 
the dynamics
 \begin{equation}\label{semilinear intro}
\dot x(t)=Ax(t)+f_{\s(t)}(x(t),u(t)), \quad t\geq 0,
\end{equation}
where $x(t)\in X$, $U$ is a Banach space, $u\in U$ is the input, $A$ is the infinitesimal generator of a $C_0$-semigroup of bounded linear operators $(T_t)_{t\geq 0}$ on $X$, 
%$\mathcal{Q}$ is a nonempty index set, 
$\s:\R_+\to \mathcal{Q}$ is a piecewise constant switching function, and $f_{q}:X\times U\to X$
%, for each index $q\in \mathcal{Q}$, 
is a %uniformly 
Lipschitz continuous nonlinear operator, uniformly with respect to  $q\in \mathcal{Q}$. Assume that only discrete output measurements are available    
\begin{equation}
y(t)=x(t_k),\quad \forall~t\in [t_k,t_{k+1}),\; \forall~k\geq 0,
\end{equation}
where $(t_k)_{k\ge 0}$ denotes the increasing sequence of sampling times.
It is well known that, in general, no feedback of the type $u(t)=K(y(t))$ stabilizes system~\eqref{semilinear intro}. 
Moreover, suppose that system~\eqref{semilinear intro} in closed-loop with 
 \begin{equation}\label{feedback intro}
 u(t)=K(x(t)), \quad \forall~t\geq 0,
 \end{equation}
where $K:X\to U$ is a globally Lipschitz function satisfying $K(0)=0$, is uniformly semi-globally exponentially stable. Using our converse Lyapunov theorem, we show that if the  maximal sampling period is small enough
then, under some additional conditions,
% on the estimated norm of the transition map of system~\eqref{semilinear intro}-\eqref{feedback intro}, 
system~\eqref{semilinear intro} in closed-loop with the predictor-based sampled-data control  
\begin{equation}\label{predictor intro}
u(t)=T_{t-t_k}y(t_k),\quad \forall~t\in [t_k,t_{k+1}),\;\forall~k\geq 0, 
\end{equation}
is uniformly locally exponentially stable in each ball around the origin. 
Furthermore, if the closed loop system~\eqref{semilinear intro}-\eqref{feedback intro} is uniformly globally exponentially stable, then the same property holds for the closed loop system~\eqref{semilinear intro}-\eqref{predictor intro}, under sufficiently small sampling period. We give an example of a wave equation~(see, e.g., \cite{Chitour19, Martinez2000}) showing the applicability of our result.

In recent years, the problem of characterizing input-to-state stability (ISS) for infinite-dimensional systems has attracted a particular attention. 
Roughly speaking, the ISS property, introduced in~\cite{Sontag} for ordinary differential equations, means that the trajectories of a perturbed system eventually approach a neighborhood of the origin whose size is proportional to the magnitude of the perturbation. This concept has been widely studied in the framework of complex systems such as switching systems (see, e.g., \cite{MANCILLAAGUILAR200147} and references therein), time-delay systems (see, e.g., \cite{PEPE20061006, 701099, Yeganefar} and references therein), and abstract infinite-dimensional systems (see, e.g.,~\cite{MiroPrieur2019, MIRONCHENKO201864}). For example, in~\cite{MIRONCHENKO201864} a converse Lyapunov theorem characterizing the input-to-state stability of a locally Lipschitz dynamics through the existence of a locally Lipschitz continuous coercive ISS Lyapunov functional is given. 
 Recently in~\cite{JMPW19} it has been shown  that, under regularity assumptions on the dynamics, the existence of non-coercive Lyapunov functionals implies input-to-state stability. Here, we provide a result of ISS type, proving that the input-to-state map has finite gain, under the assumption that the unforced system corresponding to~\eqref{semilinear intro} (i.e., with $u\equiv 0)$ is uniformly globally exponentially stable. 
 
% then~\eqref{semilinear intro} is uniformly $L^p$-exponentially input-to-state stable, for every $1\leq p\leq +\infty$, i.e., the input-output map is well defined as a map from $L^p(U)$ to $L^p(X)$ and has a finite $L^p$-gain independent of $\s$.

The paper is organized as follows. Section~\ref{sec:problem state} presents the problem statement with useful notations and definitions. In Section~\ref{main results sec} we state our main results, namely three Datko-type theorems for uniform local, semi-global, and global exponential stability, together with direct and converse Lyapunov theorems. In Section~\ref{sec: discussion} we compare the proposed Lyapunov theorems with the current state of art.  The applications are given in Section~\ref{sec: applications}. In Section~\ref{sec-example} we consider an example of a damped wave equation. The proofs are postponed to Section~\ref{s:proofs}.

\subsection{Notations}\label{sec:notations}
By $(X,\|\cdot\|)$ we denote a Banach space with norm $\|\cdot\|$ and by $B_{X}(x,r)$ the closed ball in $X$ of center $x\in X$ and radius $r$. By $\R$ we denote the set of real numbers and by $|\cdot|$ the Euclidean norm of a real vector.
We use $\R_+$ and $\R_+^{\star}$ to denote the sets of non-negative and positive real numbers respectively.
A function $\alpha:\mathbb{R}_{+}\to\mathbb{R}_{+}$ is said to be of
class $\mathcal{K}$ if it is continuous, increasing, and satisfies $\alpha(0)=0$; it is said to be
of class $\mathcal{K}_{\infty}$ if it is of class $\mathcal{K}$ and unbounded. A continuous function
$\kappa:\mathbb{R}_{+}\times\mathbb{R}_{+}\to\mathbb{R}_{+}$ is said to be of class $\mathcal{KL}$ if $\kappa(\cdot,t)$ is of class
$\mathcal{K}$ for each $t\geq 0$ and, for each $s\geq0$, $\kappa(s,\cdot)$ is nonincreasing and converges to zero as $t$ tends to $+\infty$.

%Given a continuous unbounded function $\beta:\mathbb{R}_{+}\to\mathbb{R}_{+}$ with $\beta(0)=0$,
%we set
%\begin{align}
%\beta_{\sharp}(\tau):=\inf\left\{\rho>0:\beta(\rho)= \tau\right\}, \quad \forall \tau\in \mathbb{R}_+,\label{setvalueinf}\\
%\beta^{\sharp}(\tau):=\sup\left\{\rho>0:\beta(\rho)=\tau\right\}, \quad \forall \tau\in \mathbb{R}_+.\label{setvaluesup}
%\end{align}
%Notice that $\beta_{\sharp}=\beta^{\sharp}=\beta^{-1}$ if $\beta$ is  increasing.
%In the general case, 
%both $\beta_{\sharp}$ and $\beta^{\sharp}$ are nondecreasing, $\beta\circ \beta_\sharp(\tau)=\beta\circ \beta^\sharp(\tau)=\tau$, and 
%$\beta_\sharp\circ \beta(\rho)\le \rho\le \beta^\sharp\circ \beta(\tau)$.
%In particular, 
%$\beta_{\sharp}$ and $\beta^{\sharp}$ satisfy the properties
%\begin{align}
%\label{eq:psudoinf}
%\tau\le \beta(\rho)\quad \Longrightarrow\quad  \beta_{\sharp}(\tau)\le \rho,\\
%\label{eq:psudosup}
%\beta(\rho)\le \tau\quad \Longrightarrow\quad  \rho\le \beta^{\sharp}(\tau).
%\end{align}

\section{Problem statement}\label{sec:problem state}
In this paper we consider a forward complete dynamical system evolving in a  Banach space $X$. 
%We are interested in studying its different form of exponential stability, more precisely the property of local, semi-global and global exponential stability. 
Let us recall %some useful and needed definitions. 
%By the following we recall 
the 
following
definition,
% of %an abstract 
%forward complete dynamical system 
%in a Banach space $X$ (see
proposed in~\cite{MIRONCHENKO2018}. 
\begin{definition}\label{FC} 
Let $\Q$ be a nonempty set. Denote by $\S$ a set of functions
%switching signals\footnote{\paolo{set of signals/functions? (better to avoid to speak about switch for the general class?)} \yacine{Functions}} 
$\s:\R_+\to \Q$  satisfying the following conditions: 
\begin{itemize}
\item [a)] $\S$ is closed by time-shift, i.e., for all $\s\in \S$ and all $\tau\geq 0$, the $\tau$-shifted function $\mathbb{T}_{\tau}\s:s\mapsto\s(\tau+s)$
belongs to $\S$;
\item[b)] $\S$ is closed by concatenation, i.e., for all $\s_1,\s_2\in \S$ and all $\tau>0$ the function $\s$ defined by $\s\equiv\s_1$ over $[0,\tau]$ and by $\s(\tau+t)=\s_2(t)$ for all $t>0$,
belongs to $\S$.
\end{itemize}
Let $\phi: \R_+\times X\times \S\to X$ be a map. The triple $\Sigma=(X,\S,\phi)$ is said to be a  \emph{forward complete dynamical system} if
the following properties hold:
\begin{itemize}
%\item[i)] $\forall~(x,\s)\in X \times \S$ and $\forall~t\geq 0$, the value $\phi(t,x,\s)$ is well-defined in $X$; 
\item[i)] $\forall~(x,\s)\in X \times \S$, it holds that $\phi(0,x,\s)=x$;
\item [ii)] $\forall~(x,\s)\in X \times \S$, $\forall~t\geq 0$, and $\forall~\tilde\s\in\S$ such $\tilde\s=\s$ over $[0,t]$, it holds that 
 $\phi(t,x,\tilde\s)=\phi(t,x,\s)$; 
 \item [iii)] $\forall~(x,\s)\in X \times \S$, the map $t\mapsto \phi(t,x,\s)$ is continuous;
 \item [iv)] $\forall t,\tau\geq 0$, $\forall~(x,\s)\in X \times \S$, it holds that $\phi(\tau,\phi(t,x,\s),\mathbb{T}_{t}\s)=\phi(t+\tau,x,\s)$. 
\end{itemize}
We will refer to $\phi$ as the \emph{transition map} of $\Sigma$. 
\end{definition}

Observe that if $\Sigma$ is a forward complete dynamical system and $\S$ contains a constant function $\s$
then $(\phi(t,\cdot,\s))_{t\geq 0}$ is a strongly continuous nonlinear semigroup, whose definition is recalled below.

\begin{definition}\label{Semigroup}
Let $T_t:X\to X$, $t\geq 0$, be a family of nonlinear maps. We say that $(T_t)_{t\geq 0}$ is a strongly continuous nonlinear semigroup if
the following properties hold:
\begin{itemize}
\item[i)] $\forall~x\in X$,  $T_0 x=x$; 
\item [ii)] $\forall~t_1,t_2\geq 0$, $T_{t_1}T_{t_2}x=T_{t_1+t_2}x$; 
 \item [iii)] $\forall~x\in X$, the map $t\mapsto T_{t}x$ is continuous. 
\end{itemize}
\end{definition}

An example of forward complete dynamical system is given next.
\begin{example}[Piecewise constant switching system]\label{PC}
We denote by $\mathrm{PC}$ the set of piecewise constant $\s:\R_+\to \mathcal{Q}$, and we consider here 
the case $\S=\mathrm{PC}$. 
Let $\s\in \mathrm{PC}$ be constantly equal to $\s_k$ over $[t_{k}, t_{k+1})$, with $0=t_{0}<t_{1}<\cdots <t_{k}<t<t_{k+1}$, for $k\geq 0$. With each $\s_k$ we associate the 
strongly continuous nonlinear semigroup $(T_{\s_k}(t))_{t\geq 0}:=(\phi(t,\cdot,\s_k))_{t\geq 0}$. By concatenating the flows 
$(T_{\s_k}(t))_{t\geq 0}$, one can associate with $\s$ the family of nonlinear evolution operators 
\begin{equation*}\label{sigalotti1}
T_{\s}(t):= T_{\s_k}(t-t_{k})T_{\s_{k-1}}(t_{k}-t_{k-1})\cdots T_{\s_1}(t_{1}), 
\end{equation*}
$t\in [t_{k},t_{k+1})$. 
By consequence, system $\Sigma$ can be identified with the piecewise constant switching system
\begin{equation}\label{sigalotti}
%\begin{cases}
x(t)=T_{\s}(t)x_0, \ x_{0}\in X, \ \s\in \mathrm{PC}.
%x_{0}\in X.
%\end{cases}
\end{equation}
\end{example}
Thanks to the representation given by~\eqref{sigalotti}, this paper extends  to the nonlinear case
some of the results obtained in~\cite{Hante2011} on the 
characterization of the
exponential stability of switching linear systems in Banach spaces. 

Various notions of uniform (with respect to the functions in $\S$) exponential stability of system $\Sigma$ are given by the following definition.
\begin{definition}\label{0-GES def}
Consider the forward complete dynamical system $\Sigma=(X,\S,\phi)$. 
\begin{enumerate}
\item  We say that $\Sigma$ is 
\emph{uniformly globally exponentially stable at the origin} (UGES, for short) 
if there exist %positive reals 
$M>0$ and $\lambda>0$ such that the transition map $\phi$ satisfies the inequality 
\begin{equation*}
\|\phi(t, x, \s)\|\leq M e^{-\lambda t} \|x\|,
\ \forall~t\geq 0,\; \forall~x\in X,\; \forall~\s\in\S.
\end{equation*}
%for every  $t\geq 0$, $x\in X$, and $\s\in\S$.
\item We say that $\Sigma$ is 
\emph{uniformly locally exponentially stable at the origin} (ULES, for short) 
if there exist %positive real numbers 
$R>0$, $M>0$, and $\lambda>0$ such that the transition map $\phi$ satisfies the inequality, 
for every  $t\geq 0$, $x\in B_X(0,R)$, and $\s\in\S$, 
\begin{equation}\label{les-def}
\|\phi(t, x, \s)\|\leq M e^{-\lambda t} \|x\|.
% \quad \forall~t\geq 0,\; \forall~x\in B_{X}(0,R),\; \forall~\s\in\S.
\end{equation}
If inequality~\eqref{les-def} holds true for a given $R>0$ then we say that $\Sigma$ is 
\emph{uniformly exponentially stable at the origin in $B_{X}(0,R)$} (UES in $B_{X}(0,R)$, for short). 
\item We say that $\Sigma$ is 
\emph{uniformly semi-globally exponentially stable at the origin} (USGES, for short) 
if, for every $r>0$ there exist $M(r)>0$ and $\lambda(r)>0$ such that
the transition map  $\phi$ satisfies the inequality 
\begin{equation}\label{sges-def}
\|\phi(t, x, \s)\|\leq M(r) e^{-\lambda(r) t} \|x\|,%  \quad \forall~t\geq 0,\; \forall~x\in B_{X}(0,r),\; \forall~\s\in\S.
\end{equation}
for every  $t\geq 0$, $x\in B_X(0,r)$, and $\s\in\S$.
\end{enumerate}
\end{definition}

%The notion of uniform local exponential stability is recalled by the following definition.
%\begin{definition}\label{0-GES def}
%We say that the system $\Sigma=(X,\S,\phi)$ is ULES if, there exist positive real numbers $R$, $M$ and $\lambda$ such that the transition map $\phi$ satisfies the inequality 
%\begin{equation}\label{les-def}
%\|\phi(t, x, \s)\|\leq M e^{-\lambda t} \|x\|,  \quad \forall~t\geq 0, \forall~x\in B_{X}(0,R), \forall~\s\in\S.
%\end{equation}
%\ih{If inequality~\eqref{les-def} holds true for a given $R>0$ then we say that $\Sigma$ is UES in $B_{X}(0,R)$. }
%%If~\eqref{sges-def} holds true for every $r>0$ and for some $M=M(r)$ and $\lambda=\lambda(r)$ then $\Sigma=(X,\S,\phi)$ %is called USGES.
%\end{definition}

%The notion of uniform semi-global exponential stability is recalled by the following definition.
%\begin{definition}\label{0-SGES def}
%We say that the system $\Sigma=(X,\S,\phi)$ is USGES if, for any positive real number $r$ there exist positive reals $M(r)$ and $\lambda(r)$ such that
%the flow $\phi$ satisfies the inequality 
%\begin{equation}\label{sges-def}
%\|\phi(t, x, \s)\|\leq M(r) e^{-\lambda(r) t} \|x\|,  \quad \forall~t\geq 0, \forall~x\in B_{X}(0,r), \forall~\s\in\S.
%\end{equation}
%\end{definition}

\begin{remark}\label{remark-USGES}
Up to modifying $r\mapsto M(r)$, one can assume without loss of generality in the definition of  USGES that
$r\mapsto \lambda(r)$ can be taken %positive and 
constant and $r\mapsto M(r)$ nondecreasing.
Indeed, let us fix $M:=M(1)$ and $\lambda:=\lambda(1)$.
One has, by definition of $(M,\lambda)$,
\begin{equation*}
\|\phi(t, x, \s)\|\leq Me^{-\lambda t} \|x\|, 
% \quad \forall~t\geq 0,\; \forall~x\in B_{X}(0,1),\; \forall~\s\in\S.
\end{equation*}
for every  $t\geq 0$, $x\in B_X(0,1)$, and $\s\in\S$.
For $R>1$,  by using~\eqref{sges-def} there exists $t_{R}$ such that, for every $x\in B_X(0,R)$ with $\|x\|\geq 1$, one has for $t\geq t_{R}$,
\begin{equation*}
M(R)e^{-\lambda(R)t_{R}}R=1,~ \|\phi(t,x,\s)\|\leq\frac{\|x\|}{R}\le \min\{1,\|x\|\}.
\end{equation*}
%and 
%\begin{equation*}
%\|\phi(t,x,\s)\|\leq\frac{\|x\|}{R}\le \min\{1,\|x\|\} , \quad \forall t\geq t_{R}. 
%\end{equation*}
This implies that, for every  $t\geq t_R$, $x\in B_X(0,R)$, and $\s\in\S$,
\begin{equation*}
\|\phi(t, x, \s)\|\leq Me^{-\lambda (t-t_{R})} \|\phi(t_{R},x,\s)\|\leq 
Me^{-\lambda (t-t_{R})}\|x\|.
%  \quad \forall~t\geq t_{R},\; \forall~x\in B_{X}(0,R),\; \forall~\s\in\S.
\end{equation*}
By setting 
\begin{equation*}
\widehat M(R)=\begin{cases}
M &R\leq 1,\\
\max\{M(R),M(R)e^{|\lambda-\lambda(R)|t_{R}},Me^{\lambda t_R}\}&R>1,
\end{cases}
%\widehat M(R)=\left\{
%\begin{array}{lll}
%M & R\leq 1,\\
%\max\{M(R),M(R)e^{|\lambda-\lambda(R)|t_{R}},Me^{\lambda t_R}\}, & R>1,
%\end{array}\right.
\end{equation*}
one has that  
%\begin{equation*} 
$\|\phi(t, x, \s)\|\leq \widehat M(r) e^{-\lambda t} \|x\|, $
% \quad \forall~t\geq 0,\; \forall~x\in B_{X}(0,r),\; \forall~\s\in\S.
%\end{equation*}
for every  $t\geq 0$, $x\in B_X(0,r)$, and $\s\in\S$.
Finally, we may replace $r\mapsto \widehat M(r)$ with the nondecreasing function $r\mapsto \inf_{\rho\geq r}\widehat M(\rho)$.
\end{remark}

The property of semi-global exponential stability introduced in Definition~\ref{0-GES def} 
turns out to be satisfied by 
some interesting 
class of infinite-dimensional systems, as described in the following two examples.

%\begin{example}\label{example wave without switch}
%Let $\Omega$ be a bounded open domain of 
%class $\C^2$ in $\R^n$, and consider the  (distributed)\footnote{\mario{speak of internal damping?}} damped wave equation 
%\begin{equation}\label{wave without switch}
%\left\{\begin{array}{lll}
%{\psi}_{tt}-\Delta \psi+\rho(t,{\psi}_{t})=0\quad &(x,t)\in\Omega\times\R_+,\\
%\psi=0\quad & (x,t)\in\partial\Omega\times\R_+,\\
%\psi(0)=\psi_0, \psi_t(0)=\psi_1\quad & x\in\Omega,
%\end{array}\right.
%\end{equation}
%where % $\psi$ belongs to some functional space and 
%$\rho:\R_+\times \R\to\R$ is a sufficiently regular nonlinear function. % with some regularity. 
%Depending on the behaviour of the nonlinear function $\rho$, different types of stability (polynomial and global exponential ones) have been studied in the literature (for polynomial and global exponential stability, see, e.g.,~\cite{Chitour19, Martinez2000}). \paolo{specify $X$?}
%  Depending on the behavior of the nonlinear function $\rho$, different types of stability (polynomial, global exponential, and semi-global exponential stability \yacine{attention, il faut passer en controle frontiere}) have been studied in the literature . 
%For instance, in the case where $\sigma$ is a saturation function constraining the amplitude of the feedback control, %it is well known that 
% \mario{the} 1D wave damped equation~\eqref{boundary-wave without switch} is at best semi-globally exponentially stable, see \cite{Martinez2000}. 
%\end{example}

\begin{example}\label{example kdv without switch-old}
For $L>0$, let $\Omega=(0,L)$ and consider the controlled Korteweg--de Vries (KdV) equation 
\begin{equation}\label{kdv}
\begin{cases}
{\eta}_{t}+{\eta}_{x}+{\eta}_{xxx}+ \eta{\eta}_{x}+\rho(t,x,\eta)=0 & (x,t)\in \Omega\times\R_+,\\
\eta(t,0)=\eta(t,L)={\eta}_{x}(t,L)=0 & t\in \R_+,\\
\eta(0,x)={\eta}_0(x) & x\in \Omega,
\end{cases}
\end{equation}
where %$\eta$ belongs to some functional space and 
$\rho:\R_+\times\Omega\times \R\to\R$ is a sufficiently regular nonlinear function. % with some regularity. 
The case $\rho\equiv 0$ is a well known model describing waves on shallow water surfaces~\cite{ISRAWI2019854}. The controllability and stabilizability properties of~\eqref{kdv} have been extensively studied in the literature (see, e.g.,~\cite{Menzala2002, Rosier2009}). In the case where the feedback control is of the form 
$\rho(t,x,\eta)=a(x)\eta$, for some non-negative function $a(\cdot)$ having nonempty support in $\Omega$, 
%the forward complete dynamical 
system% corresponding to
~\eqref{kdv} is globally exponentially stable in $X=L^2(0,L)$. In~\cite{MCPA17} the authors prove that, when a saturation is introduced in the feedback control $\rho$, the system is only semi-globally exponentially stable in $X$.
\end{example}

\begin{example}\label{example boundary wave without switch}
Consider the 1D wave equation with boundary damping  
\begin{equation}\label{boundary-wave without switch}
\begin{cases}
{\psi}_{tt}-\Delta \psi=0 &(x,t)\in(0,1)\times\R_+,\\
\psi(0,t)=0&t\in\R_+,\\
\psi_x(1,t)=-\sigma(t,\psi_t(1,t)) & t\in\R_+,\\
\psi(0)=\psi_0, \psi_t(0)=\psi_1 & x\in(0,1),
\end{cases}
\end{equation}
where $\sigma:\R_+\times \R\to\R$ is continuous.
This system is of special interest when 
%dealing with controlled linear wave equation, where 
the damping term $\sigma(t,\psi_t(1,t))$ represents a nonlinear feedback control. Once again, different types of stability  can be established (for global and semi-global exponential stability, see e.g.,~\cite{Marx20, bastin2016stability}). In particular, if  
 $\sigma$ is a nonlinearity of saturation type, only semi-global exponential stability holds true in
 \[X=\{(\psi_0,\psi_1)\mid \psi_0(0)=0, ~\psi_0'~ \mbox{and}~\psi_1\in L^{\infty}(0,1)\}.\] 
 \end{example}

%{\color{orange} [OLD EXAMPLES]
%%\begin{example}\label{example wave without switch-old}
%%Let $\Omega$ be a bounded open domain of 
%%class $\C^2$ in $\R^n$, and consider the  damped wave equation 
%%\begin{equation}\label{wave without switch-old}
%%\left\{\begin{array}{lll}
%%{\psi}_{tt}-\Delta \psi+\rho(t,{\psi}_{t})=0\quad &(x,t)\in\Omega\times\R_+,\\
%%\psi=0\quad & (x,t)\in\partial\Omega\times\R_+,\\
%%\psi(0)=\psi_0, \psi_t(0)=\psi_1\quad & x\in\Omega,
%%\end{array}\right.
%%\end{equation}
%%where % $\psi$ belongs to some functional space and 
%%$\rho:\R_+\times \R\to\R$ is a sufficiently regular nonlinear function. % with some regularity. 
%%This system is of special interest when dealing with controlled linear wave equation, where the damping term $\rho(t,{\psi}_{t})$ represents the applied nonlinear feedback control. Depending on the behavior of the nonlinear function $\rho$, different types of stability (polynomial, global exponential, and semi-global exponential stability) have been studied in the literature (see, e.g.,~\cite{Chitour19, Martinez2000}). For instance, in the case where $\rho$ is a saturation function constraining the amplitude of the feedback control, %it is well known that 
%% \mario{the} 1D wave damped equation~\eqref{wave without switch-old} is at best semi-globally exponentially stable \mario{[CITATION FOR THIS PRECISE RESULT - Theorem?]}. 
%% \paolo{specify $X$?}
%%\end{example}
%

%Of course, 
The systems considered in Examples~\ref{example kdv without switch-old} and~\ref{example boundary wave without switch} do not depend on $\sigma\in\mathcal{S}$ as in Definition~\ref{FC}.
%are non-switching systems.\footnote{\paolo{To avoid speaking  of switch in relation with forward complete dynamical systems it is maybe better to write something like ``The systems considered in Examples~\ref{example wave without switch} and~\ref{example kdv without switch} .''}}
In Section~\ref{sec-example} we %justify and motivate the consideration 
introduce and study variants
of such systems 
in a switching framework.

\section{Main results}\label{main results sec}

\subsection{Datko-type theorems}
In this section we give Datko-type theorems~\cite{Datko} 
for an
%in the framework  of the 
abstract forward complete dynamical system $\Sigma$. 
The uniform (local, semi-global, and global) exponential stability is characterized in terms of the $L^p$-norm of the trajectories of the system. This provides a generalization of the results obtained in~\cite{Ichikawa} for nonlinear semigroups. 

The following theorem characterizes the local exponential stability of system $\Sigma$. 

\begin{thm}\label{ULES}
Consider a forward complete dynamical system $\Sigma=(X,\S,\phi)$. 
Let $t_1, G_0>0$, and $\beta$ be a function of class $\mathcal{K}_{\infty}$
%nonnegative continuous function defined on $\R_+$ 
such that %$\beta(0)=0$, 
$\limsup_{r\downarrow 0} \frac{\beta(r)}{r}$ is finite and 
 \begin{equation}\label{exp-bounded2}
\|\phi(t,x,\s)\| \leq G_0\beta(\|x\|), \quad \forall~t\in[0,t_1],\; \forall~x\in X,\; \forall~\s\in\S. 
\end{equation}
The following statements are equivalent 
\begin{description}
\item [i)] System $\Sigma$ is ULES;
\item [ii)] for every $p>0$ there exist a nondecreasing function $k:\mathbb{R}_+\to \mathbb{R}_+$ and $R>0$ such that 
\begin{equation}\label{exp-lp3}
\int_{0}^{+\infty}{\|\phi(t,x,\s)\|}^{p} dt \leq k(\|x\|)^p\|x\|^{p}, 
%\quad \forall~x\in B_X(0,R),\; \forall~\s\in\S;
\end{equation}
for every $x\in B_X(0,R)$ and $\s\in\S$;
\item [iii)] there exist $p>0$, $k:\mathbb{R}_+\to \mathbb{R}_+$ nondecreasing, and $R>0$ such that~\eqref{exp-lp3} holds true. 
\end{description}
\end{thm}
\begin{remark}
 Observe that hypothesis~\eqref{exp-bounded2} in Theorem~\ref{ULES} is global over $X$. 
 Indeed, we do not know if the stability at 0 may be deduced from inequality~\eqref{exp-lp3} if one restricts~\eqref{exp-bounded2} to a ball $B_X(0,r)$.
% Indeed, if we restrict~\eqref{exp-bounded2} to a ball $B_X(0,r)$, then we are not able to deduce from inequality~\eqref{exp-lp3} 
%the stability at 0. Actually, it seems reasonable that counterexamples exist. 
\end{remark}

The following theorem characterizes the semi-global exponential stability of system $\Sigma$. 

\begin{thm}\label{USGES}
Consider a forward complete dynamical system $\Sigma=(X,\S,\phi)$. 
Let $t_1, G_0>0$, and $\beta$ be a function of class $\mathcal{K}_{\infty}$
%nonnegative continuous function defined on $\R_+$ 
such that %$\beta(0)=0$, 
$\limsup_{r\downarrow 0} \frac{\beta(r)}{r}$ is finite and 
 \begin{equation}
\|\phi(t,x,\s)\| \leq G_0\beta(\|x\|), \quad \forall~t\in[0,t_1],\; \forall~x\in X,\; \forall~\s\in\S. 
\end{equation}
The following statements are equivalent 
\begin{description}
\item [i)] System $\Sigma$ is USGES;

\item [ii)] for every $p>0$ there exists a nondecreasing function $k:\mathbb{R}_+\to \mathbb{R}_+$ such that,
for every $x\in X$ and $\s\in\S$ 
\begin{equation}\label{lp0}
\int_{0}^{+\infty}{\|\phi(t,x,\s)\|}^{p} dt \leq  k(\|x\|)^p\|x\|^{p};
\end{equation}

\item [iii)] there exist $p>0$ and $k:\mathbb{R}_+\to \mathbb{R}_+$ nondecreasing such that~\eqref{lp0} holds true. 
\end{description}
\end{thm}

The particular case of uniformly globally exponentially stable systems is considered in the following theorem.
%The characterization of the uniform global exponential stability of system $\Sigma$ is given by the following theorem as a particular case of Theorem~\ref{USGES}. 
\begin{thm}\label{datko}
Consider a forward complete dynamical system $\Sigma=(X,\S,\phi)$. 
Let $t_1>0$ and $G_0>0$ be such that 
\begin{equation}\label{exp-bounded}
 \|\phi(t,x,\s)\| \leq G_0\|x\|, \ \forall~t\in [0,t_1],\; \forall~x\in X,\; \forall~\s\in\S. 
\end{equation}
The following statements are equivalent 
\begin{description}
\item [i)] System $\Sigma$ is UGES;
\item [ii)] for every $p>0$ there exists $k>0$ such that
\begin{equation}\label{unasoladai} 
\int_{0}^{+\infty}{\|\phi(t,x,\s)\|}^{p} dt \leq k^{p}\|x\|^{p},%\ \forall~x\in X,\ \forall~\s\in\S;
\end{equation}
for every $x\in X$ and $\s\in\S$;
\item [iii)]  there exist $p,k>0$ such that \eqref{unasoladai} holds true.
%\begin{equation} 
%\int_{0}^{+\infty}{\|\phi(t,x,\s)\|}^{p} dt \leq k^{p}\|x\|^{p}, \quad \forall~x\in X, \forall~\s\in\S.
%\end{equation}
\end{description}
\end{thm}

\begin{remark}
By the shift-invariance properties given by items a) and iv) of Definition~\ref{FC}, 
it is easy to see that~\eqref{exp-bounded} implies
\begin{equation}\label{rem-paolo}
\|\phi(t,x,\s)\| \leq Me^{\lambda t}\|x\|,\ \forall~t\geq 0,\; \forall~x\in X,\; \forall~\s\in\S,
\end{equation}
where $M=G_0$ and $\lambda = \max\{0,\log (\frac{G_0}{t_1})\}$. Notice that inequality~\eqref{rem-paolo} is a nontrivial
requirement  
on system $\Sigma$. Even in the linear case, and even if~\eqref{rem-paolo} is satisfied for each constant $\s\equiv \s_c$, uniformly with respect to $\s_c$, it does not follow that a similar exponential bound holds for the corresponding 
system $\Sigma$ (see~\cite[Example~1]{Hante2011}).
 \end{remark}

\subsection{Lyapunov characterization of exponential stability}\label{lyap-sec}
In this section we characterize the exponential stability of a forward complete dynamical system 
%$\Sigma$ 
through the existence of a Lyapunov functional.
First, let us recall the  definition of Dini derivative of a %continuous 
functional $V:X\to \mathbb{R}_+$. 

\begin{definition}\label{Dini derivative}
Consider a forward complete dynamical system $\Sigma=(X,\S,\phi)$. 
The upper and lower Dini derivatives  $\overline{D}_{\s}V:X\to 
\R\cup\{\pm \infty\}$ and $\underline{D}_{\s}V:X\to \R\cup\{\pm \infty\}$ of a  
functional $V:X\to \mathbb{R}_+$
are defined, respectively,  as 
\begin{equation*}
\overline{D}_{\s}V(x)=\limsup_{h\downarrow 0}\frac{1}{h}\left(V(\phi(h,x,\s))-V(x)\right),
%\quad  \forall~x\in X,\; \forall~\s\in\S,
\end{equation*}
and 
\begin{equation*}
\underline{D}_{\s}V(x)=\liminf_{h\downarrow 0}\frac{1}{h}\left(V(\phi(h,x,\s))-V(x)\right),
%\quad  \forall~x\in X,\; \forall~\s\in\S.
\end{equation*}
where $x\in X$ and $\s\in\S$.
\end{definition}

\begin{remark}\label{rem-dini-constant}
When $\S$ contains $\mathrm{PC}$, we can associate with every $q\in \Q$ the upper and lower Dini derivatives $\overline{D}_{q}V$
and $\underline{D}_{q}V$ corresponding to $\s\equiv q$. Notice that for every 
$\s\in \mathrm{PC}$ and sufficiently small $h>0$, we have $\s_{|_{(0,h)}}\equiv q$, for some $q\in\mathcal{Q}$. By consequence, we have $\overline{D}_{\s}V(\varphi)=\overline{D}_{q}V(\varphi)$ and $\underline{D}_{\s}V(\varphi)=\underline{D}_{q}V(\varphi)$. 
\end{remark}

The regularity of a Lyapunov functional associated with an exponentially stable 
forward complete dynamical
system $\Sigma=(X,\S,\phi)$ is recovered, in our results, from %depends on 
the regularity of the transition map $\phi$. 
The $\S$-uniform continuity of the transition map $\phi$ with respect to the initial condition is defined as follows.
\begin{definition}
We say that the transition map $\phi$ of $\Sigma=(X,\S,\phi)$ is $\S$-uniformly continuous if, for any $\bar t>0$, $x\in X$, and $\varepsilon>0$, there exists $\eta>0$ such that
$$\|\phi(t,x,\s)-\phi(t,y,\s)\|\leq \varepsilon,
% \quad  \forall~t\in [0,\bar t],\; \forall~y\in B_X(x,\eta),\; \forall~\s\in \S.
$$
for every $t\in [0,\bar t]$, $y\in B_X(x,\eta)$, and $\s\in \S$.
\end{definition}

Similarly, the notion of $\S$-uniform Lipschitz continuity of the transition map 
is given by the following definition.
\begin{definition}\label{def:philip}
We say that the transition map $\phi$ of $\Sigma=(X,\S,\phi)$ is $\S$-uniformly Lipschitz continuous (respectively, $\S$-uniformly Lipschitz continuous on bounded sets) if, for any $\bar t>0$ (respectively, $\bar t>0$ and $R>0$), there exists $l(\bar t)>0$ (respectively, $l(\bar t,R)>0$) such that  
\[\|\phi(t,x,\s)-\phi(t,y,\s)\|\leq l(\bar t)\|x-y\|, 
%\quad \forall~t\in [0,\bar t],\;\forall x,y \in X, \;\forall~\s\in \S
\]
for every $t\in [0,\bar t]$, $x,y \in X$, and $\s\in \S$ 
(respectively,
\[\|\phi(t,x,\s)-\phi(t,y,\s)\|\leq l(\bar t,R)\|x-y\|,
% \quad \forall~t\in [0,\bar t], \;\forall x,y\in B_{X}(0,R), \;\forall~\s\in \S,
\]
for every $t\in [0,\bar t]$, $x,y \in B_X(0,R)$, and $\s\in \S$). 
\end{definition}

The following theorem %is a direct result showing 
shows that the existence of a non-coercive Lyapunov functional is sufficient for proving the uniform exponential stability of the forward complete dynamical system $\Sigma$, provided that inequality~\eqref{exp-bounded2} holds true.  
\begin{thm}\label{ULES-SG}
Consider a forward complete dynamical system $\Sigma=(X,\S,\phi)$. 
Let $t_1, G_0>0$, and $\beta$ be a function of class $\mathcal{K}_{\infty}$
%nonnegative continuous function defined on $\R_+$ 
such that %$\beta(0)=0$, 
$\limsup_{r\downarrow 0} \frac{\beta(r)}{r}$ is finite and 
 \begin{equation}\label{exp-bounded-bis}
\|\phi(t,x,\s)\| \leq G_0\beta(\|x\|), \quad \forall~t\in[0,t_1],\; \forall~x\in X,\; \forall~\s\in\S.
\end{equation}
Then,
\begin{itemize}
\item[i)] if there exist $R>0$, $V: B_{X}(0,R)\to \R_+$, and $p, c>0$ 
such that, for every $x\in B_X(0,R)$ and $\s\in \S$, 
\begin{align}
V(x)&\leq c\|x\|^{p},  \label{noncoe1-SG}\\
\underline{D}_{\s}V(x)&\leq -\|x\|^{p},\label{dini1-SG}
\end{align}
and $V(\phi(\cdot,x,\s))$ is continuous from the left at  every $t>0$ such that $\phi(t,x,\s)\in B_{X}(0,R)$,
then system $\Sigma$ is ULES;
\item [ii)] 
%By a slight abuse of notation, set
%\begin{equation}\label{setvalue}
%\beta^{-1}(\tau):=\sup\left\{\rho>0:\beta(\rho)\leq \tau\right\}, \quad \forall \tau\geq 0.
%\end{equation}
if i) holds true for every $R>0$ with $V=V_R$, $c=c_R$ and $p=p_R$ and 
\begin{equation}\label{Pettersen}
\limsup_{R\to +\infty}\beta^{-1}\left(\frac{R}{G_0}\right)\min\left\{1,\left(\frac{t_1}{c_R}\right)^{\frac{1}{p_R}}\right\} =+\infty,
\end{equation} 
%where $\beta^{\sharp}$ is defined as in \eqref{setvaluesup},
%\begin{equation}\label{Pettersen}
%\limsup_{R\to +\infty}\sup_{\tau>0}\left\{\tau:\beta(\tau)\leq \frac{R}{G_0}\right\}\min\left\{1,\left(\frac{t_1}{c_R}\right)^{\frac{1}{p_R}}\right\} =+\infty,
%\end{equation} 
then system $\Sigma$ is USGES;
\item [iii)] 
if $\beta$ in~\eqref{exp-bounded-bis} is equal to the identity function and there exist 
$p,c>0$ and 
a functional $V: X\to \R_+$ %and positive reals $p$ and $c$ 
such that, for every $x\in X$ and $\s\in \S$, 
the map $t\mapsto V(\phi(t,x,\s))$ is continuous from the left,
%\footnote{\ih{is this really important?: \textcolor{blue}{probably not}
%what type of regularity on $V$ which guarantee the left continuity of the composition?: non, nothing less than continuity: \textcolor{blue}{patchy Lyapunov, for example $\widetilde V=V+1_{\{x:V(x)\geq R\}}$}}} 
and $V$ satisfies inequalities~\eqref{noncoe1-SG}-\eqref{dini1-SG} in $X$, then system $\Sigma$ is UGES.
\end{itemize}
\end{thm}

%\textcolor{red}{
%\begin{remark}
%remplacer \eqref{exp-bounded2} par there exists $r>0$ such that
%$$
%\|\phi(t+s,x,\s)\| \leq g(t)\beta(\|\phi(s,x,\s)\|), \quad \forall~t\in [0,t_1], \forall s\geq 0, \forall~x\in B_r(0), \forall~\s\in\S;
%$$
%\end{remark}
%}

The following theorem %is a converse result stating 
states that the existence of a coercive Lyapunov functional is necessary for the
uniform exponential stability of a forward complete dynamical system. % $\Sigma$.  
 \begin{thm}\label{coercive1}
 Consider a forward complete dynamical system $\Sigma=(X,\S,\phi)$ and assume that the transition map $\phi$ is $\S$-uniformly continuous. If $\Sigma$ is USGES then, for every $r>0$ there exist $\underline{c}_r,\overline{c}_r>0$ and a continuous functional $V_r:  X\to \R_+$,  such that 
\begin{align}
\label{norm comparison2}
\underline{c}_r\|x\|\leq V_r(x)\leq \overline{c}_r\|x\|,& \quad \forall~x\in B_X(0,r), \\
\label{dini9}
\overline{D}_{\s}V_r(x)\leq -\|x\|,& \quad \forall~x\in B_X(0,r),\; \forall~\s\in\S,\\
V_r=V_R\mbox{ on }X, &\quad \forall R>0\mbox{ such that }\lambda(r)=\lambda(R)\nonumber\\
&\quad\mbox{ and }M(r)=M(R),\label{eq:indis}
\end{align}
where $\lambda(\cdot), M(\cdot)$ are as in~\eqref{sges-def}.
Moreover, in the case where the transition map $\phi$ is $\S$-uniformly Lipschitz continuous (respectively, 
$\S$-uniformly Lipschitz continuous on bounded sets), 
%there exist a positive real $L_r$ such that 
$V_r$ %is $L_r$-
can be taken Lipschitz continuous (respectively, Lipschitz continuous on bounded sets).
\end{thm}
% Theorem~\ref{coercive1} is stated for a uniformly semi-globally exponentially stable system $\Sigma$. In particular, the theorem holds true when $\Sigma$ is uniformly globally exponentially stable.  
%
%
To conclude this section, we state the following corollary which characterizes the uniform global exponential stability
of a forward complete dynamical system, completing Item iii) of Theorem~\ref{ULES-SG}.
%system $\Sigma$. 

\begin{cor}\label{main}
Consider a forward complete dynamical system $\Sigma=(X,\S,\phi)$. Assume that the transition map $\phi$ is $\S$-uniformly continuous. If there exist $t_1>0$ and $G_0>0$ such that
 \begin{equation}\label{exp-bound cor} 
\|\phi(t,x,\s)\| \leq G_0\|x\|, \quad \forall~t\in[0,t_1],\; \forall~x\in X,\; \forall~\s\in\S,
\end{equation}
then the following statements are equivalent:
\begin{itemize}
\item [i)] System $\Sigma$ is UGES;
\item [ii)] there exists a continuous functional $V: X\to \R_+$ and positive reals $p$, $\underline c$, and $\overline c$ such that 
\begin{equation*}\label{norm comparison3}
\underline c\|x\|^{p}\leq V(x)\leq \overline c\|x\|^{p}, \quad \forall~x\in X, 
\end{equation*}
and 
\begin{equation}\label{dini10}
\overline{D}_{\s}V(x)\leq -\|x\|^{p}, \quad \forall~x\in X,\; \forall~\s\in\S;
\end{equation}
\item [iii)] there exist a functional $V: X\to \R_+$ and positive reals $p$ and $c$ such that, for every $x\in X$ and $\s\in \S$, the map $t\mapsto V(\phi(t,x,\s))$ is continuous from the left,  inequality~\eqref{dini10} is satisfied and the following inequality holds
\begin{equation*}\label{norm comparison3-nonco}
V(x)\leq c\|x\|^{p}, \quad \forall~x\in X. 
\end{equation*} 
\end{itemize}  
\end{cor}

\begin{proof}
The fact that item i) implies ii) is a straightforward consequence of Theorem~\ref{coercive1}, using, in particular, \eqref{eq:indis}. 
%[REALLY? from its proof, maybe. But then we should move also this prove to Section 7]} \paolo{[Should work now with the new statement of Theorem~\ref{coercive1} 
%\yacine{ Je ne vois pas l'argument juste a partir de l'enonce} 
 Moreover ii) clearly implies iii).
Finally, iii) implies i), as follows from Theorem~\ref{ULES-SG}.
\end{proof}

\section{Discussion: comparison with the current state of the art}\label{sec: discussion}
We compare here the results stated in Section~\ref{lyap-sec} with %to 
some interesting similar results, obtained recently in~\cite{Mironchenko2019}, concerning the Lyapunov characterization of the uniform global asymptotic stability % characterization 
of a forward complete dynamical system. % $\Sigma$. 
In order to make this comparison, we briefly recall some definitions and assumptions from~\cite{Mironchenko2019}. 
%The following definition is a standard definition of uniform global asymptotic stability of a forward complete dynamical system. % $\Sigma$. 
\begin{definition}\label{0-GAS def}
We say that a forward complete dynamical system $\Sigma=(X,\S,\phi)$ is 
\emph{uniformly globally asymptotically stable at the origin} (UGAS, for short) if there exists a function $\kappa$ of class $\mathcal{KL}$ such that the transition map $\phi$ satisfies the inequality 
\begin{equation*}
\|\phi(t, x, \s)\|\leq \kappa(\|x\|,t),  \quad \forall~t\geq 0,\; \forall~x\in X,\; \forall~\s\in\S.
\end{equation*}
\end{definition}
The notion of \emph{robust forward completeness} of system $\Sigma$ is given by the following. 
\begin{definition}
The forward complete dynamical system $\Sigma=(X,\S,\phi)$ is said to be \emph{robustly forward complete} (RFC, for short) if for any $C>0$ and any $\tau>0$
it holds that 
\begin{equation*}\label{RFC}
\sup_{\|x\|\leq C, t\in [0,\tau], \s\in \S}\|\phi(t,x,\s)\|<\infty. 
\end{equation*}
\end{definition}
Notice that RFC property of $\Sigma$ is equivalent to inequality
\eqref{exp-bounded-bis}, although it does not necessarily imply that $\limsup_{r\downarrow 0}\frac{\beta(r)}{r}$ is finite.

The notion of \emph{robust equilibrium point}, which 
may be seen as a form of weak stability at the origin,
%is a form of ``weak stability" property, 
is given as follows.
\begin{definition}\label{REP}
We say that $0\in X$ is a \emph{robust equilibrium point} (REP, for short) of the forward complete dynamical system $\Sigma=(X,\S,\phi)$ if
%it is an equilibrium point such that 
for every $\varepsilon, h>0$,  there exists $\delta=\delta(\varepsilon,h)>0$, so that 
\begin{equation*}
%t\in [0,h], \|x\|\leq \delta, \s\in \S\implies \|\phi(t,x,\s)\|\leq \varepsilon. 
\|x\|\leq \delta\implies \|\phi(t,x,\s)\|\leq \varepsilon,\quad \forall t\in [0,h], \forall  \s\in \S.
\end{equation*}
\end{definition}
One of the main results
%The main result 
obtained in~\cite{Mironchenko2019} relates the UGAS property with %states that 
the existence of a non-coercive Lyapunov functional, i.e., a continuous function $V:X\to\R_+$ satisfying $V(0)=0$ %with 
and the two inequalities
\begin{equation*}\label{norm comparison2-GAS}
0< V(x)\leq \alpha_1(\|x\|), \quad \forall~x\in X\backslash\{0\}, 
\end{equation*}
and 
\begin{equation*}\label{dini9-GAS}
\overline{D}_{\s}V(x)\leq -\alpha_2(\|x\|), \quad \forall~x\in X,\; \forall~\s\in\S,
\end{equation*}
where $\alpha_1\in\mathcal{K}_{\infty}$ and 
$\alpha_2\in \mathcal{K}$.
%, conjointly with the RFC and REP
%properties imply that system $\Sigma$ is UGAS. 
This is formulated by the following theorem.
\begin{thm}[\cite{Mironchenko2019}]\label{Miro19}
Consider a forward complete dynamical system $\Sigma=(X,\S,\phi)$ and assume that $\Sigma$ is robustly forward complete and that
$0$ is a robust equilibrium point of $\Sigma$. If 
$\Sigma$ admits a  non-coercive Lyapunov functional,  
then it is UGAS.  
\end{thm}

%From Theorem~\ref{Miro19}, 
The existence of a non-coercive %uniform 
Lyapunov functional $V$ is not sufficient in order to get the uniform global asymptotic stability of system $\Sigma$. Indeed, as shown in~\cite[Example 6.1]{Mironchenko2019}, the existence of such a $V$ does not imply either the RFC or REP property, hence the necessity of both 
assumptions.
%conditions. 
Even in the linear case, an RFC like condition is required (see~\cite[Remark 4]{Hante2011}).

Let us compare Theorem~\ref{Miro19} with our corresponding result for UGES. As we already noticed, RFC is equivalent to \eqref{exp-bounded-bis}, and in particular it is ensured by the stronger condition \eqref{exp-bound cor}. 
Moreover, it is easy to check that  \eqref{exp-bound cor} also implies the REP property.
 Thus, the hypotheses of Theorem~\ref{ULES-SG} imply those of Theorem~\ref{Miro19}. As a counterpart, a stronger stability property is obtained  (namely, UGES). Also notice that our theorem %present some 
relaxes the requirements on the functional $V$, since the lower Dini derivatives, instead of the upper one, is used and discontinuities of $V$ along the trajectories of $\Sigma$ are allowed.

%In our case, the robust forward completeness property required by Theorem~\ref{Miro19} is a direct consequence of the hypotheses of Theorem~\ref{ULES-SG}. Indeed, as it is shown in the proof of Theorem~\ref{ULES-SG}, conditions~\eqref{exp-bounded-bis} together with~\eqref{noncoe1-SG}--\eqref{dini1-SG} lead to the extension of condition~\eqref{exp-bounded-bis}, robustly with respect to $\s$, {\color{red} over whole the interval $[0,+\infty)$}. Also, in our case,  the REP property is not required and this for the simple reason that the strong conditions required on $V$ (the functions $\alpha_1$ and $\alpha_2$ in~\eqref{norm comparison2-GAS} and~\eqref{dini9-GAS} are {\color{red} power functions}) ensure that this property is automatically %intrinsically 
%satisfied.
%{\color{red} Note further that the condition required on $\beta$, more precisely on its limit superior at zero, is essential in order to get an exponential bound leading, together with~\eqref{noncoe1-SG}--\eqref{dini1-SG}, to the uniform exponential stability. 

\section{Applications}\label{sec: applications}
\subsection{Nonlinear retarded systems with piecewise constant delays}
Let $r\geq 0$ and set $\C=\C([-r,0],\mathbb{R}^n)$, the set of continuous functions from $[-r,0]$ to $\R^n$. 
Consider the nonlinear retarded system 
%\begin{eqnarray}\label{delay}
%\dot x(t)&=&F(x(t-\tau(t))), \quad t\geq 0,\\
%x(\theta)&=&\varphi(\theta), \quad \theta\in [-r,0],
%\end{eqnarray}
%where $x(t)\in \R^n$, $\varphi\in \C$, $\tau:\R_+\to [0,r]$ be a piecewise continuous delay function 
%and $F:\R^n\to\R^n$ a continuous function. 
\begin{equation}
\begin{array}{ll}
\dot x(t)=f_{\s(t)}(x_t), & \quad t\geq 0,\label{delay1}\\
x(\theta)=\varphi(\theta), &\quad \theta\in [-r,0],%\label{delay0}
\end{array}
\end{equation}
where $x(t)\in \R^n$, $\varphi\in \C$, $x_t:[-r,0]\rightarrow\R^n$
is the standard notation for the history function defined by 
\begin{equation*}
x_t(\theta)=x(t+\theta), \quad -r\leq\theta\leq0,
\end{equation*}
$\s:\R_+\to \mathcal{Q}$ is a piecewise constant function, and $f_{q}:\C\to \R^n$ is a continuous functional such that 
$f_{q}(0)=0$ for all $q\in \mathcal{Q}$.

For every $q\in \mathcal{Q}$ and $\varphi\in \C$, we assume that
there exists a unique solution $x$ over $[-r,+\infty)$ of~\eqref{delay1} with $\s(t)=q$ for every $t\geq 0$.  
This defines a family $(T_{q}(t))_{t\geq 0}$ of nonlinear maps from $\C$ into itself by setting 
\begin{equation*}\label{delay-free0}
T_{q}(t)\varphi =x_t, 
\end{equation*}
for $t\geq 0$. %Actually, 
According to \cite{Hale,WEBB1974}, $(T_{q}(t))_{t\geq 0}$ is a strongly continuous semigroup 
of nonlinear operators on $\C$. 
%We denote by $\mathrm{PC}$ the set of piecewise constant functions $\s:\R_+\to \mathcal{Q}$.
%For every $\s\in \mathrm{PC}$, there exist an increasing sequence of times $(t_k)_{k\geq 0}$
%and a sequence $(q_k)_{k\geq 0}$ taking values in  $\mathcal{Q}$ such that 
%$t_0=0$, $\lim_{k\to\infty}t_k=\infty$ and  $\s$ is constantly equal to $q_k$ on $[t_k,t_{k+1})$, for every $k\geq 0$.
%By concatenating the flows $(T_{q_k}(t))_{t\geq 0}$ 
%%and following the same reasoning as in Example~\ref{PC}
%we identify the dynamics of system~\eqref{delay1} with the transition map 
%\begin{equation}\label{delay-free}
%\begin{array}{lll}
%\phi(t,\varphi,\s)=T_{\s}(t)\varphi,
%\end{array}
%\end{equation}
%and 
We denote by $\Sigma_r=(\C,\mathrm{PC},\phi)$ the corresponding forward complete dynamical system constructed as in Example~\ref{PC}.

As a consequence of the switching representation of the nonlinear time-varying delay system~\eqref{delay1}, the results of the previous section (in particular Theorem~\ref{ULES-SG} and Corollary~\ref{main}) apply to system $\Sigma_r$. Let us explicitly provide an application of Corollary~\ref{main}.

\begin{thm} 
Let $L>0$ be such that 
\begin{equation}\label{lip}
|f_{q}(\psi_1)-f_{q}(\psi_2)|\leq L\|\psi_1-\psi_2\|, \ \forall~\psi_1, \psi_2\in \C,\;\forall~q\in\Q.
\end{equation}
The following statements are equivalent:
\begin{itemize}
\item [i)] System $\Sigma_r$ is UGES;
\item [ii)] there exists a continuous functional $V: X\to \R_+$ and positive reals $p$, $\underline c$, and $\overline c$ such that 
\begin{equation*} 
\underline c\|\psi\|^{p}\leq V(\psi)\leq \overline c\|\psi\|^{p}, \quad \forall~\psi\in \C,
\end{equation*}
and 
\begin{equation}\label{dini-delay}
\overline{D}_{q}V(\psi)\leq -\|\psi\|^{p}, \quad \forall~\psi\in \C,\; \forall~q\in \Q;
\end{equation}
\item [iii)] there exist a functional $V: \C\to \R_+$ and positive reals $p$ and $c$ such that, for every $\psi\in \C$ and $q\in \Q$, the map $t\mapsto V(T_{q}(\cdot)\psi)$ is continuous from the left,  inequality~\eqref{dini-delay} is satisfied, and %the following inequality holds
\begin{equation*} 
V(\psi)\leq c\|\psi\|^{p}, \quad \forall~\psi\in \C.
\end{equation*} 
\end{itemize}  
\end{thm}

\begin{proof}

The proof is based on Corollary~\ref{main}. In order to apply it, we have to prove that the transition map is 
$\mathrm{PC}$-uniformly continuous and that~\eqref{exp-bound cor} holds true. Using \eqref{lip}, we easily get
\begin{align*} 
\|\phi(t,\varphi_1,\sigma)-&\phi(t,\varphi_2,\sigma)\|\leq L\int_{0}^{t}{\|\phi(s,\varphi_1,\sigma)-\phi(s,\varphi_2,\sigma)\|ds}, \quad \forall~t\geq 0,   
\end{align*}
which implies 
\begin{equation*} 
\|\phi(t,\varphi_1,\sigma)-\phi(t,\varphi_2,\sigma)\|\leq e^{Lt}\|\varphi_1-\varphi_2\|, \quad \forall~t\geq 0.   
\end{equation*}
Hence, the transition map $\phi$ is $\mathrm{PC}$-uniformly Lipschitz continuous and, using the fact that $f_q(0)=0$ for every $q\in\Q$, we have that~\eqref{exp-bound cor} holds true with $G_0=e^{Lt_1}$. 
%The result then follows from Corollary~\ref{main}. 
%In particular since $0$ is an equilibrium it follows that $\|x_t\|\leq l(\bar t)\|\varphi\|$ where $l(\bar t)$ is given as in Definition~\ref{def:philip} so that~\eqref{exp-bound cor} holds true. The result then follows from Corollary~\ref{main}.
\end{proof}

%\textcolor{red}{
%\subsection{Semi global exponential stability of interconnected nonlinear switching systems}
%%. A. Loria and A. Chaillet. Kristin Y. Pettersen (Lyapunov sufficient conditions for USGES-Automatica 2017)
% ...
%\subsection{Coercive Lyapunov characterization of input-to-state stability for linear switching systems} 
%%
%\begin{example}
%\begin{itemize}
%\item Interconnected (cascade) PDEs 1-$D$ wave functions with switching 
%\item Wave equation with slowly ($\dot \alpha(t)<constant$) moving domain  and saturation (see Tr\'elat paper 2017)
%\end{itemize}
%\end{example}
%}

\subsection{Predictor-based sampled data exponential stabilization}\label{sample}
Let $X, U$ be two Banach spaces. Consider the semilinear control system 
\begin{equation}\label{sample1}
\left\{\begin{array}{l}
\dot x(t)=Ax(t)+f_{\s(t)}(x(t),u(t)), \quad t\geq 0,\\
x(0)=x_0\in X,
\end{array}\right.
\end{equation}
where $u\in \C(\R_+,U)$ is the control input, $A$ is the infinitesimal generator of a $C_0$-group 
$(T_{t})_{t\in\R}$
of bounded linear operators  on $X$, $\s:\R_+\to \mathcal{Q}$ is a piecewise constant function, and $f_{q}:X\times U\to X$, for $q\in \mathcal{Q}$, is a Lipschitz continuous nonlinear operator, with Lipschitz constant $L_f>0$ independent of $q$, such that $f_{q}(0,0)=0$. Let $K:X\to U$ be a globally Lipschitz function with Lipschitz constant $L_K>0$, satisfying $K(0)=0$, and consider system~\eqref{sample1} in closed-loop with 
 \begin{equation}\label{feedback}
 u(t)=K(x(t)), \quad \forall~t\geq 0.
 \end{equation}
Observe that, since $A$ is the infinitesimal generator of a $C_0$-group, the following inequality holds for the corresponding induced norm: there exist $\Gamma ,\omega>0$ such that 
\begin{equation}\label{Groupe}
\|T_{t}\|\leq \Gamma e^{\omega |t|}, \quad \forall~t\in \R.  
\end{equation}
%for some $\Gamma ,\omega>0$.
%\begin{ass}\label{C0group}
%A is the generator of a $C_0$-group of bounded linear operators $(T(t))_{t\in \R}$. 
%\end{ass}

%Note, from~\cite{Pazy}, that Assumption~\ref{C0group} is true if $0$ belongs to the resolvent set associated with %$T(t_0)$, for some $t_0>0$, and inequality~\eqref{Groupe} holds. 

%Recall the following theorem from~\cite{Pazy}.
%\begin{thm}{\cite[Theorem 1.4+Theorem 1.4, Chapter 6]{Pazy}}\label{existence+uniqueness}
%Let $f:X\times [0,+\infty)\to X$ be continuous in $t$ for $t\geq 0$ and locally Lipschitz continuous in $x$, uniformly in $t$ on bounded intervals.
%If $A$ is the infinitesimal generator of a $C_0$-semigroup $T(t)$ on $X$ then for every $x_0\in X$ there is a $t_{\max}\leq +\infty$ such that 
%the initial value problem
%\begin{eqnarray}\label{initial value problem}
%\dot x(t)&=&Ax(t)+f(x(t),t), \quad t\geq 0,\\
%x(0)&=&x_0
%\end{eqnarray}
%has a unique mild solution on $[0,t_{\max})$, given by
%\begin{equation}\label{mild}
%x(t)=T(t)x_0+\int_{0}^{t}T(t-s)f(x(s),s)ds, \quad t\in [0,t_{\max}).
%\end{equation}
%Moreover, if $t_{\max}<+\infty$ then $\lim_{t\uparrow t_{\max}}\|x(t)\|=+\infty$. In addition, 
%for every $T>0$, the mapping $x_0\to x$ is Lipschitz continuous from $X$ into $C([0,T],X)$. 
%\end{thm}

%\begin{ass}\label{feedback}
%There exists a globally Lipschitz feedback $K:X\to U$, with Lipschitz constant $L_K>0$, satisfying $K(0)=0$, such that the %closed-loop system~\eqref{sample1} with the state feedback $u(t)=K(x(t))$ is USGES.  
%\end{ass} 

 The aim of this section is to show the usefulness of our converse theorems in the study of 
 exponential stability preservation under sampling for the semilinear control 
 switching system~\eqref{sample1}. For convenience of the reader we give the following definition. 
 
\begin{definition}
Let $0<s_1<\cdots <s_k<\cdots$ be an increasing sequence of times such that $\lim_{k\to+\infty}s_k=+\infty$. 
The instants $s_k$ are called \emph{sampling instants} 
%if the designed feedback controller is implemented 
%at just the discret instants $s_k$, for $k\geq 0$. 
and the quantity 
\begin{equation*}\label{sample def}
\delta=\sup_{k\geq 0}(s_{k+1}-s_k)
\end{equation*}
is called the \emph{maximal sampling time}. By \emph{predictor-based sampled data controller}
we mean a feedback $u(\cdot)$ of the type 
\begin{equation}\label{eq:PBSDC}
u(t)=K(T_{t-s_k}x(s_k)), \quad \forall~t\in [s_k,s_{k+1}),\; \forall~k\geq 0. 
\end{equation}
\end{definition}

\subsubsection{Converse Lyapunov result for the closed-loop system~\eqref{sample1}--\eqref{feedback}} 
For every $\s\in \mathrm{PC}$, there exist an increasing sequence of times $(t_k)_{k\geq 0}$
and a sequence $(q_k)_{k\geq 0}$ taking values in  $\mathcal{Q}$ such that 
$t_0=0$, $\lim_{k\to\infty}t_k=\infty$ and  $\s$ is constantly equal to $q_k$ on $[t_k,t_{k+1})$ for every $k\geq 0$.
For every $q\in \mathcal{Q}$ and $x_0\in X$, 
letting $\s(t)\equiv q$,
there exists a unique mild solution of~\eqref{sample1} over $[0,+\infty)$,  
%which is given by~(see, e.g., \cite{Pazy})
i.e., a continuous function $x(\cdot,x_0,q)$ satisfying
\begin{align*} 
x(t,x_0&,q)=T_{t}x_0+\int_{0}^{t}T_{t-s}f_{q}\left(x(s,x_0, q),K(x(s,x_0, q))\right)ds, %\quad t\geq 0.
\end{align*}
for every $t\ge 0$.
 This defines a family $(T_{q}(t))_{t\geq 0}$ of nonlinear maps by setting $T_{q}(t)x_0 =x(t,x_0,q)$,
for $t\geq 0$. We denote by $\Sigma_0=(X,\mathrm{PC},\phi)$ the corresponding forward complete dynamical system constructed as in  Example~\ref{PC}.  By \cite[Corollary~1.6]{Frankowska}, system $\Sigma_0$ is $\mathrm{PC}$-uniformly Lipschitz continuous. As a consequence of Theorem~\ref{coercive1}, we have the following lemma.
\begin{lem}\label{coercive3}
Suppose that $\Sigma_0$ is USGES. Then, for every $r>0$, there exist $\underline{c}_r,\overline{c}_r>0$ and a Lipschitz continuous Lyapunov functional $V_r: X\to \R_+$ such that 
\begin{equation*} 
\underline{c}_r\|x\|\leq V_r(x)\leq \overline{c}_r\|x\|, \quad \forall~x\in B_X(0,r), 
\end{equation*}
and 
\begin{equation*} 
\overline{D}_{q}V_r(x)\leq -\|x\|, \quad \forall~x\in B_X(0,r),\; \forall~q\in\Q.
\end{equation*}
%where $r\mapsto M(r)$ is associated 
% with $\Sigma_0$ as in Definition~\ref{0-GES def}.
\end{lem}

\subsubsection{Predictor-based sampled data feedback}
Consider the predictor-based sampled data switching control system 
\begin{equation}\label{sample2}
\left\{\begin{array}{l}
\dot x(t)=Ax(t)+f_{\s(t)}(x(t),K(T_{t-s_k}x(s_k))),\ s_{k}\leq t<s_{k+1},\\
x(0)=x_0,
\end{array}\right.
\end{equation}
where $(s_k)_{k\geq 0}$ is the increasing sequence of sampling instants. 
We will say that $x^{\Sigma}:\R_+\to X$ is a solution of system~\eqref{sample2} if $t\mapsto x^{\Sigma}(t)$ is continuous and for every $k\geq 0$ and $s_{k}\leq t<s_{k+1}$ one has 
\begin{align}
x^{\Sigma}&(t)=T_{t-s_k}x^{\Sigma}(s_k)+\int_{0}^{t-s_k}T_{t-s_k-s}f_{\s(s+s_k)}\left(x^{\Sigma}(s+s_k),K\left(T_{s}x^{\Sigma}(s_k)\right)\right)ds,\label{sample-mild}
\end{align}
that is, the restriction of $x^{\Sigma}$ on $[s_k,s_{k+1}]$ is a mild solution.   

By applying~\cite[Theorem 6.1.2]{Pazy} on every interval $[s_k,s_{k+1}]$, we deduce the following.
\begin{lem}
For every $x_0\in X$ and $\s\in\mathrm{PC}$, system~\eqref{sample2} admits a unique solution $x^{\Sigma}:\R_+\to X$.
%, i.e., a continuous function $x^{\Sigma}(\cdot)$ obtained by solving recursively
%% given by the recursive formula 
%\begin{equation}\label{sample-mild}
%x^{\Sigma}(t)=T(t-s_k)x^{\Sigma}(s_k)+\int_{0}^{t-s_k}T(t-s_k-s)f_{\s(s+s_k)}\left(x^{\Sigma}(s+s_k),K\left(T(s)x^{\Sigma}(s_k)\right)\right)ds, \quad s_{k}\leq t<s_{k+1},
%\end{equation}
%for $k\geq 0$. 
\end{lem}

For $x_0\in X$ and $\s\in \mathrm{PC}$, following the same reasoning as before, we identify the dynamics of system~\eqref{sample2} with the transition map
\begin{equation*} 
\phi(t,x_0,\s)=x^{\Sigma}(t), \quad t\geq 0,
\end{equation*}
where $x^{\Sigma}$ is given by~\eqref{sample-mild}, and we denote by $\Sigma=(X,\mathrm{PC},\phi)$ the corresponding forward complete dynamical system.  Notice that $\Sigma$ depends on the sequence of sampling instants $(s_k)_{k\in\mathbb{N}}$.

\begin{thm}\label{sampling theorem}
%If $\Sigma_0$ is ULES, then 
%there exists %a positive number 
%$\delta^{\star}>0$ such that $\Sigma$ is 
%ULES, provided that the maximal sampling time of $(s_k)_{k\in\mathbb{N}}$
%%$\delta$ 
%is smaller than  $\delta^{\star}$.
%Moreover, 
If system~$\Sigma_0$ is USGES and
\begin{equation}\label{eq:limMr}
\lim_{r\to \infty}\frac{r}{M(r)}=+\infty,
\end{equation}
where $M(r)$ is as in \eqref{sges-def},
then for every $r>0$ there exists 
$\delta^{\star}(r)>0$ such that $\Sigma$ is 
UES in $B_{X}(0,r)$, provided that the maximal sampling time of $(s_k)_{k\in\mathbb{N}}$
%$\delta$ 
is smaller than  $\delta^{\star}(r)$.
\end{thm}

When $\Sigma_0$ is uniformly globally exponentially stable, the associated Lyapunov functional is globally uniformly Lipschitz, and by consequence (see the proof of Theorem~\ref{sampling theorem}), the following corollary 
holds true.
%is  a straightforward consequence of Theorem~\ref{sampling theorem}. 

\begin{cor}\label{sampling theorem-UGES}
Suppose that system~$\Sigma_0$ is UGES.
Then there exists $\delta^{\star}>0$ such that 
system~$\Sigma$ is UGES provided that the maximal sampling time of $(s_k)_{k\in\mathbb{N}}$
%$\delta$ 
is smaller than  $\delta^{\star}$.
\end{cor}

%{\ihab{Do you think that is important:
%Consider the perturbed sampled data feedback system (from a practical point of vue: when dealing with feedbcack stabilization under perturbed discret output measurements) 
%\begin{equation}\label{sample2-pert}
%\left\{\begin{array}{lll}
%\dot x(t)&=&Ax(t)+f_{\s(t)}(x(t),K(T(t-s_k)x(s_k))+\varepsilon(t)), \quad s_{k}\leq t<s_{k+1},\\
%x(0)&=&x_0,
%\end{array}\right.
%\end{equation}
%where $\varepsilon$ is an $L^p$ perturbation ...
%\begin{cor}\label{sampling theorem-UGES-robustness}
%Suppose that system~$\Sigma_0$ is UGES.
%Then there exists $\delta^{\star}>0$ such that 
%system~\eqref{sample2-pert} is ISS (in the sens of inequality~\eqref{lp1}) provided that the maximal sampling time of $(s_k)_{k\in\mathbb{N}}$
%%$\delta$ 
%is smaller than  $\delta^{\star}$.
%\end{cor}
%If yes, we have to interchange subsections. 
%}}

%\ih{\begin{remark}
%It seems natural that in the case of USGES the sampling period depends on the norm of the initial condition
%and may converge to zero as $r$ tends to $+\infty$. For example, consider the system in $\R^2$ described in polar coordinates by 
%\begin{align}\label{paolo}
%&\dot r=-f(r)\\
%&\dot \theta=1, 
%\end{align}
%with $f$ a positive bounded real function such that $\lim_{r\to +\infty}f(r)=0$ in such a way that the solution of \eqref{paolo} tends slowly to zero. Then the trajectories at $x$ become closely tangent to the circle $C(0,\|x\|)$ when $\|x\|$ tends to $+\infty$. This implies that the sampling time need to be shorter and shorter. 
%\end{remark}}

\subsection{Link between uniform global exponential stability and uniform input-to-state stability}
Let $X$ and $U$ be two Banach spaces and consider the control system~\eqref{sample1}, where $A$ is the infinitesimal generator of a $C_0$-semigroup of bounded linear operators $(T_{t})_{t\geq 0}$ 
on $X$, $\s\in \mathrm{PC}$, and $f_{q}:X\times U\to X$, for $q\in \mathcal{Q}$, is a Lipschitz continuous nonlinear operator, with Lipschitz constant $L_f>0$ independent of $q$, such that $f_{q}(0,0)=0$. 
We assume here that the set of admissible controls is $L^p(U):=L^p(\R_+,U)$ with $1\leq p\leq +\infty$. Following the reasoning in Section~\ref{sample} we can define (see, e.g., \cite{Frankowska}), for every $x_0\in X$, $\s\in \mathrm{PC}$, and $u\in L^p(U)$, the corresponding trajectory $%t\mapsto 
\phi_u(t,x_0,\s)$ on $\R_+$,  which is 
%well defined  as an 
absolutely continuous %function 
with respect to $t$ and continuous with respect to $(x_0,u)\in X\times L^p(U)$. 

We next provide a result of ISS type in the same spirit as  those obtained in~\cite{JMPW19, MIRONCHENKO201864}. In our particular context (UGES and global Lipschitz assumption) we are able to prove that the input-to-state map $u\mapsto \phi_u(\cdot,0,\s)$ has finite gain. 

\begin{thm}\label{iss theorem}
Assume that the forward complete dynamical system $(X,%\S
\mathrm{PC},\phi_0)$ is UGES. Then for every $1\leq p\leq +\infty$ and $\s\in \mathrm{PC}$, the 
input-to-state map $u\mapsto \phi_u(\cdot,0,\s)$ is well defined as a map from $L^p(U)$ to $L^p(X)$ and has a finite $L^p$-gain independent of $\s$, i.e., 
there exists $c_p>0$ such that 
\begin{equation}\label{lp1}
\|\phi_u(\cdot,0,\s)\|_{L^p(X)}\leq c_p\|u\|_{L^p(U)}, \quad \forall~u\in L^p(U),\; \forall~\s\in\mathrm{PC}.
%\S. 
\end{equation}
\end{thm}

\section{Example: Sample-data exponential stabilization of a switching %damped 
wave equation}\label{sec-example}
Let $\Omega$ be a bounded open domain of 
class $\C^2$ in $\R^n$, and consider the switching damped wave equation 
\begin{equation}\label{wave}
\left\{\begin{array}{ll}
\frac{\partial^2\psi}{\partial t^2}-\Delta x+\rho_{\s(t)}(\frac{\partial\psi}{\partial t})=0\quad &\textnormal{in $\Omega\times\R_+$},\\
\psi=0\quad & \textnormal{on $\partial\Omega\times\R_+$},\\
\psi(0)=\psi_0, 
\psi^{\prime}(0)=\psi_1\quad & \textnormal{on $\Omega$},
\end{array}\right.
\end{equation}
where $\s:\R_+\to \mathcal{Q}$ is a piecewise constant function and $\rho_q:\R\to\R$, for $q\in \mathcal{Q}$, is a uniformly Lipschitz continuous function satisfying 
%the following properties
\begin{equation*}
\rho_q(0)=0, \quad  \alpha |v|\leq |\rho_q(v)|\leq \frac{|v|}{\alpha}, \quad \forall~v\in\R,\; \forall~q\in \mathcal{Q},
\end{equation*}
for some $\alpha>0$. 
In the case where $\tilde\rho(t,v):=\rho_{\s(t)}(v)$ is sufficiently regular, namely  a continuous function differentiable on $\R_+\times (-\infty,0)$ and $\R_+\times (0,\infty)$, and $v\mapsto \tilde\rho(t,v)$ is nondecreasing, for each initial condition $(\psi_0,\psi_1)$ taken in 
$H^2(\Omega)\cap H^1_0(\Omega)\times H^1_0(\Omega)$ there exists a unique strong solution for~\eqref{wave} in the class
%\begin{equation}\label{class}
%\psi\in 
${\bf H}=W^{2,\infty}_{\rm loc}(\R_+,L^2(\Omega))\cap W^{1,\infty}_{\rm loc}(\R_+,H^1_0(\Omega))\cap L^{\infty}_{\rm loc}(\R_+,H^2(\Omega)\cap H^1_0(\Omega))
$
%\end{equation}
(see \cite{Martinez2000} for more details).
For the switching damped wave equation~\eqref{wave} the existence and uniqueness of a strong solution (in
${\bf H}$) 
%the sense of \eqref{class}) 
is given by concatenation. Defining the energy of the solution of~\eqref{wave} by 
\begin{equation*}
E(t)=\frac{1}{2}\int_{\Omega}\left({\frac{\partial\psi}{\partial t}}^2+|\nabla \psi|^2\right)dx,
\end{equation*}
we can prove, following the same lines of the proof of~\cite[Theorem 1]{Martinez2000}, that the energy of the solutions in ${\bf H}$
%\eqref{class} 
decays uniformly (with respect to the initial condition) exponentially to zero as
\begin{equation*}
E(t)\leq E(0)\exp\left(1-\mu t\right), \quad \forall~t\geq 0, 
\end{equation*}
for some $\mu>0$ that depends only on $\alpha$. 
Let $X$ be the Banach space $H^1_0(\Omega)\times L^2(\Omega)$ 
endowed with the norm 
\begin{equation*}\label{norm}
\|x\|=\|\nabla x_1\|^2_{L^2(\Omega)}+\|x_2\|^2_{L^2(\Omega)},
\end{equation*}
and let $A$ be the linear operator defined on $X$ by 
\begin{align*}
D(A)=&\Big\{x=\left(\begin{array}{c}x_1\\x_2\end{array}\right)\in X \mid x_1\in H^2(\Omega)\cap H^1_0(\Omega),\\
&  \quad x_2\in H^1_0(\Omega)\Big\},\\
A=&
\left(\begin{array}{cc}0&I\\ \Delta & 0\end{array}\right), 
\end{align*}
where $I$ is the identity operator and $\Delta$ denotes the Laplace operator. 
%In the case of the bounded domain $\Omega$ the following defines a norm on $X$ 
It is well known that $D(A)$
 is dense in $X$ and that $A$ is the infinitesimal generator of a $C_0$-group of bounded linear operators 
$(T_{t})_{t\in \R}$ on $X$ satisfying 
%\yacine{for the induced norm} \mario{[why here instead of \eqref{Groupe}?]} 
\begin{equation*} 
\|T_{t}\|\leq \Gamma e^{|t|}, \quad \forall~t\in \R. 
\end{equation*}
With this formulation, equation~\eqref{wave} can be %transformed to 
rewritten as
the initial value problem
\begin{equation}\label{wave-cauchy} 
\left\{\begin{array}{l}
\dot x(t)=Ax(t)+f_{\s(t)}(x(t),u(t)), \quad t\geq 0,\\
x(0)=x_0,
\end{array}\right.
\end{equation}
with feedback
\begin{equation*}
u(t)= x_2(t) \textnormal{ and }  f_s(x,u)=\left(\begin{array}{c}0\\ \rho_s(u)\end{array}\right).
\end{equation*}
The associated transition map satisfies the inequality
\begin{equation}\label{decay wave}
\|\phi(t,x,\s)\|\leq e^{1-\mu t}\|x\|, \quad \forall~t\geq 0.
\end{equation}
Note that the constant $\mu$ does not depend on the solution. Using the density of $D(A)$ in $X$, inequality~\eqref{decay wave} holds true for weak solutions. 
Thus system~\eqref{wave-cauchy} is UGES. The assumptions of 
Corollary~\ref{sampling theorem-UGES}
%Theorem~\ref{sampling theorem} are 
being satisfied, % Thus, by Theorem~\ref{sampling theorem}  
we deduce that
for a sufficiently small maximal sampling time the 
predictor-based sampled data feedback
\eqref{eq:PBSDC}
 preserves the exponential
decay to zero of the energy of the solutions of~\eqref{wave-cauchy}.

\section{Proof of the main results}\label{s:proofs}
As a preliminary step, let us state the following useful and straightforward lemma. 
\begin{lem}
Given a continuous function $\beta:\R_+\to \R_+$ satisfying $\limsup_{r\downarrow 0} \frac{\beta(r)}{r}<+\infty$,
the function $\alpha:\R_+\to \R_+$ defined by 
\begin{equation}\label{alpha-beta}
\alpha(0)=1+\limsup_{r\downarrow 0} \frac{\beta(r)}{r},\ \alpha(r)=1+\sup_{s\in (0,r]}\frac{\beta(s)}{s}, \ r>0,
\end{equation}
is a continuous nondecreasing function satisfying $\beta(r)\leq r\alpha(r)$ for every $r\geq 0$.
\end{lem}

%\ih{what do you think if we remove this lemma after adding in the notation section the subclass 
%$\mathcal{K}^{\infty}\subset \mathcal{K}_{\infty}$ of functions $\beta$ satisfying $\beta(r)\leq r\alpha(r)$
%for some nondecreasing positive function $\alpha$
%and remplacing our condition in all theorems mentioning $\beta$ like the following
%\begin{thm} 
%Consider a forward complete dynamical system $\Sigma=(X,\S,\phi)$. 
%Let $t_1, G_0>0$, and $\beta$ be a function of class $\mathcal{K}^{\infty}$ such that 
% \begin{equation} 
%\|\phi(t,x,\s)\| \leq G_0\beta(\|x\|), \quad \forall~t\in[0,t_1],\; \forall~x\in X,\; \forall~\s\in\S.
%\end{equation}
%$$\vdots$$
%\end{thm}
% }
%\mario{It can be a good idea, even if it is maybe better to stick to classical classes of functions. Moreover, splitting the two conditions, we can say that one is equivalent to the RFC condition in MW. I let P and Y decide what is better.}
%\paolo{I slightly prefer to keep only the usual notations}. \yacine{I strongly prefer to keep the usual notations: our potential reviewers do that}

\subsection{Proof of Theorem~\ref{ULES}}
By using inequality~\eqref{les-def}, one deduces that i) %\textbf{i)} 
implies ii) %\textbf{ii)} 
with $r\mapsto k(r)\equiv\frac{M}{(p\lambda)^{1/p}}$.
%, for some positive reals $R, M$ and $\lambda$.
Moreover, ii) %\textbf{ii)} 
clearly implies iii). %\textbf{iii)}. 
It then remains to prove that iii) implies i).
%\textbf{iii)} implies \textbf{i)}. 
Without loss of generality, we assume that 
$G_0\geq 1$. Let $C%, \widetilde C
:\R_+\to \R_+$ be the nondecreasing function defined by 
\begin{equation}\label{C}
C(r)=G_0\max\left\{\alpha(r),\frac{k(r)}{t_1^{1/p}}\alpha\left(\frac{k(r)r}{t_1^{1/p}}\right)\right\},
% \quad \textnormal{and} \quad \widetilde C(r)=C(rC(r)),
\end{equation}
where $\alpha:\R_+\to \R_+$ is given by~\eqref{alpha-beta}.

For $t>t_1$, $x\in X$, and $\s\in\S$, we have 
\begin{align}
\|\phi(t,x,\s)\|&=\|\phi\left(\tau,\phi(t-\tau, x, \s), \mathbb{T}_{t-\tau}\s\right)\|
\leq G_0\beta\left(\|\phi(t-\tau, x, \s)\|\right), \quad \forall~\tau\in [0,t_1]. \label{bound1}
\end{align}

%Letting
% $\beta_{\sharp}$ %and $\beta^{\sharp}$ 
% be defined as in (\ref{setvalueinf}--\ref{setvaluesup}),
One deduces from equations
\eqref{exp-lp3} and \eqref{bound1} 
that, for $t>t_1$, $x\in B_X(0,R)$, and $\s\in\S$, 
\begin{align*}\label{bound2}
t_1\left(\beta^{-1}\left(\frac{\|\phi(t,x,\s)\|}{G_0}\right)\right)^{p}&\leq  \int_{t-t_1}^{t}{\|\phi(\tau,x,\s)\|^{p}d\tau}\leq  \int_{0}^{+\infty}{\|\phi(\tau,x,\s)\|^{p}d\tau}\leq k(\|x\|)^p\|x\|^p.
%,\quad t> t_1.
\end{align*}
Therefore, 
for $t>t_1$, $x\in B_X(0,R)$, and $\s\in\S$, we have 
\begin{equation}\label{betaincreasing}
\|\phi(t,x,\s)\|\leq  G_0\beta\left(\frac{k(\|x\|)\|x\|}{t_1^{1/p}}\right). %,\quad t> t_1.
\end{equation}
By consequence, 
bundling together \eqref{exp-bounded2} and \eqref{betaincreasing},  we get for $t\geq 0$ and $x\in B_X(0,R)$ that
%the following inequalities 
\begin{align}
\|\phi(t,x,\s)\|&\leq G_0\max\left\{\beta(\|x\|),\beta\left(\frac{k(\|x\|)\|x\|}{t_1^{1/p}}\right)\right\}
\leq 
C(\|x\|)\|x\|,\label{bound3}
\end{align}
where we used that $\beta(\|x\|)\le \|x\|\alpha(\|x\|)$.

Let $r>0$ be such that $rC(r)<R$. In particular, 
for every $t\ge 0$ and $x\in B_X(0,r)$ we have $\|\phi(t,x,\s)\|<R$. 
 It follows from~%\eqref{exp-lp3} together with~
\eqref{bound3} that, for every $t\ge 0$ and $x\in B_X(0,r)$,     
\begin{align*}
t\|\phi(t,x,\s)\|^{p}&=\int_{0}^{t}{\|\phi(t-\tau,\phi(\tau,x,\s),\mathbb{T}_{\tau}\s)\|^{p}d\tau}
\leq \int_{0}^{t}{ C(r)^{p}\|\phi(\tau,x,\s)\|^{p}d\tau}%\nonumber\\
%&\leq& \int_{0}^{t}{\widetilde{C}(R)^{p}\|\phi(\tau,x,\s)\|^{p}d\tau}
\leq \left( k(r) C(r) \right)^{p}\|x\|^{p},
\end{align*}
where the last inequality follows from \eqref{exp-lp3}.
By consequence, one has, for every $t\geq 0$, $x\in B_X(0,r)$, and $\s\in\S$
%the following inequality holds 
\begin{eqnarray*}
\|\phi(t,x,\s)\|\leq \frac{k(r) C\left(r\right)}{t^{1/p}}\|x\|.
% \quad \forall~t\geq 0,\;\forall~x\in B_X(0,r),\;\forall~\s\in\S.
\end{eqnarray*}
So, for each $0<c<1$, there exists a positive real number $t_0=t_0(c,r)$ such that 
\begin{eqnarray*}
\|\phi(t,x,\s)\|\leq c\|x\|, \quad \forall~t\geq t_0,\;\forall~x\in B_X(0,r),\;\forall~\s\in\S.
\end{eqnarray*}

Now, let $t\geq 0$, $x\in B_X(0,r)$, and $\s\in \S$ be fixed. There exists an integer $n\geq 0$ such that $t=nt_0+s$, with $0\leq s<t_0$. 
Notice that $\phi(jt_0,x,\s)\in B_X(0,r)$ for every $j\in \mathbb{N}$.
We have 
\begin{align*} 
\|\phi(t,x,\s)\|&=\|\phi(s,\phi(nt_0,x,\s),\mathbb{T}_{nt_0}\s)\| \leq  C(r)\|\phi(nt_0,x,\s)\|%\nonumber\\
%&\leq& \widetilde{C}(\|x\|)\|\phi(nt_0,x,\s)\| 
\leq  C(r)c^{n}\|x\|\leq M(r)e^{-\lambda(r) t}\|x\|,%\label{bound6}
\end{align*}
with $M(r)=\frac{ C(r)}{c}$ and $\lambda(r)=-\frac{\log{(c)}}{t_0(c,r)}>0$. 
The uniform local exponential stability of system $\Sigma$ is established.      

\subsection{Proof of Theorem~\ref{USGES}} 
If $\Sigma$ is USGES, then by Remark~\ref{remark-USGES} we may assume without loss of generality that $r\mapsto \lambda(r)$ is constant and $r\mapsto M(r)$ is nondecreasing. Hence, i) %\textbf{i)} 
implies ii) %\textbf{ii)} 
with $k(r):= \frac{M(r)}{(p\lambda(r))^{1/p}}$.
%By using Definition~\ref{0-GES def} one deduces
%from inequality~\eqref{sges-def}
% that \textbf{i)} implies \textbf{ii)} with 
%$k:r\mapsto \frac{M(r)}{(p\lambda(r))^{1/p}}$. 
Moreover, ii) %\textbf{ii)} 
clearly implies iii). %\textbf{iii)}.
It then remains to prove that iii) %\textbf{iii)} 
implies i). %\textbf{i)}. 
For this, let $r>0$. 
Let $R>rC(r)$, where $C$ is defined by~\eqref{C}, and observe that ~\eqref{lp0} is satisfied in $B_X(0,R)$.
% for every $R>0$, especially 
%for every $R>0$ such that $R>rC(r)$, 
 Following the same %lines of the 
 proof as in Theorem~\ref{ULES}, we get the existence of
$M(r)>0$ and $\lambda(r)>0$ such that 
\begin{equation*} 
\|\phi(t,x,\s)\|\leq M(r)e^{-\lambda(r) t}\|x\|, 
% \quad \forall~t\geq 0,\; \forall~x\in B_{X}(0,r),\; \forall~\s\in\S,
\end{equation*}
for all $t\geq 0$, $x\in B_{X}(0,r)$, and $\s\in\S$,
whence the uniform semi-global exponential stability of system $\Sigma$.     

\subsection{Proof of Theorem~\ref{datko}}
The proof follows the same steps as that of Theorem~\ref{USGES} with $\beta$ equal to the identity function
and $k$ equal to a constant function. In such a case, the functions $r\mapsto M(r)$ and $r\mapsto \lambda(r)$ are constant.

\subsection{Proof of Theorem~\ref{ULES-SG}}
We start by proving item i). 
%{\color{red} Suppose without loss of generality that $\beta$ is strictly increasing.}
Let $x\in B_X(0,R)$ and $\s\in \S$. Assume that there exists a first time $0<t^{\star}<+\infty$
such that $\|\phi(t^{\star},x,\s)\|=R$. If $t^{\star}\leq t_1$ then we have $R\leq G_0\beta(\|x\|)$
which implies that $\|x\|\geq \beta^{-1}(R/G_0)$. Then let us take $x\in B_X(0,\beta^{-1}(R/G_0))$,
so that $t^{\star}>t_1$. For every $t_1\leq t\leq t^{\star}$, by repeating the same reasoning as in~\eqref{bound1}, we have
\begin{equation}\label{bound4}
t_1\beta^{-1}\left(\frac{\|\phi(t,x,\s)\|}{G_0}\right)^{p}\leq  \int_{t-t_1}^{t}{\|\phi(\tau,x,\s)\|^{p}d\tau}. 
\end{equation}
In addition, since $t\mapsto V(\phi(t,x,\s))$ is continuous from the left, it follows from~\eqref{dini1-SG} (see~\cite[Theorem 9]{Hagood}) that  
\begin{align}
-V(x)&\leq V(\phi(t,x,\s))-V(x)\leq -\int_{0}^{t}{\|\phi(\tau,x,\s)\|^{p}d\tau}\leq -\int_{t-t_1}^{t}{\|\phi(\tau,x,\s)\|^{p}d\tau}.\label{Hagood}
\end{align}
%Inequality~\eqref{noncoe1-SG} together with~\eqref{Hagood} lead to the following 
%\begin{equation}\label{Hagood1}
%\int_{t-t_1}^{t}{\|\phi(\tau,x,\s)\|^{p}d\tau}\leq V(x). 
%\end{equation}
By consequence, from inequalities~\eqref{bound4} and \eqref{Hagood} together with~\eqref{noncoe1-SG}, one gets 
\begin{equation}\label{contradiction}
t_1\beta^{-1}\left(\frac{\|\phi(t,x,\s)\|}{G_0}\right)^{p}\leq  c\|x\|^p. 
\end{equation}
Thus, evaluating~\eqref{contradiction} at $t=t^{\star}$, one gets % the following  
\begin{equation*} 
\|x\|\geq  \left(\frac{t_1}{c}\right)^{1/p}\beta^{-1}\left(\frac{R}{G_0}\right). 
\end{equation*}
If we take $x\in B_X(0,\gamma)$ with %such that 
\begin{equation*}
\gamma:=\min\left\{\beta^{-1}\left(\frac{R}{G_0}\right),\left(\frac{t_1}{c}\right)^{1/p}\beta^{-1}\left(\frac{R}{G_0}\right)\right\} 
\end{equation*}
we get a contradiction, that is, we then have 
\begin{equation*}
\|\phi(t,x,\s)\|\leq R, \quad \forall~t\geq 0,\; \forall~x\in B_X(0,\gamma),\;\forall~\s\in \S. 
\end{equation*}
From~\eqref{dini1-SG}, we have, for every $t\geq 0$, $x\in B_X(0,\gamma)$, and $\s\in\S$,
\begin{equation*}
\underline{D}_{\s}V(\phi(t,x,\s))\leq -\|\phi(t,x,\s)\|^{p},  
%\quad \forall~t\geq 0,\; \forall~x\in B_X(0,\gamma),\; \forall~\s\in\S, 
\end{equation*}
from which, by repeating the same reasoning as in~\eqref{Hagood}, we deduce the inequality 
\begin{equation}\label{dini6}
\int_{0}^{\infty}{\|\phi(\tau,x,\s)\|^{p}d\tau}\leq c\|x\|^{p}, \ \forall~x\in B_X(0,\gamma),\;\forall~\s\in \S.
\end{equation}
Thanks to Theorem~\ref{ULES}, the uniform local exponential stability of system $\Sigma$ follows from~\eqref{dini6}
together with~\eqref{exp-bounded-bis}.

The statement ii) of the theorem follows from~\eqref{Pettersen}, \eqref{dini6}, and the definition of $\gamma$. 
The last statement follows from the fact that~\eqref{dini6}
holds true for every $x\in X$ with $c$  
independent of $\|x\|$, using Theorem~\ref{datko}.

\subsection{Proof of Theorem~\ref{coercive1}}
Fix $r>0$ and let $M=M(r)$ and $\lambda=\lambda(r)$ be as in Definition~\ref{0-GES def}. 
Choose $\gamma=\gamma(r)>0$ such that $\gamma-\lambda<0$. 
%Without loss of generality, we can suppose that $M$ is sufficiently large in such a way that we have $\gamma>1/M$. 
Let $\bar t=\frac{\log(M)}{\lambda-\gamma}.$
Let $V_r:X\to \R_+$ be the functional defined  by %as follows 
\begin{equation}\label{lyapunov3}
V_r(x)=\frac1\gamma\sup_{\s\in \S, t\in [0,\bar t]}\|e^{\gamma t}\phi(t,x,\s)\|,\quad x\in X.
\end{equation}
The functional $V_r$ is well defined, since $\Sigma$ is USGES, which implies that, 
for every $R>0$, $t\ge 0$, and $x\in B_X(0,R)$,
\begin{equation}\label{eq:decr}
\|e^{\gamma t}\phi(t,x,\s)\|\le M(R)e^{(\gamma-\lambda(R))t}\|x\|.
%\quad t\ge 0,\;x\in B_X(0,R).
\end{equation} 
Let us check inequalities~\eqref{norm comparison2}
and \eqref{dini9}. 
%For this, let $x\in B_X(0,r)$ and $\tilde\s\in\S$ be fixed. 
The right-hand inequality in~\eqref{norm comparison2} follows directly from \eqref{eq:decr}, with
\begin{equation}\label{eq:upb}
\overline{c}_r=\frac{M}{\gamma}.
\end{equation}
%the USGES assumption. 
The left-hand inequality, with 
\begin{equation}\label{eq:lob}
\underline{c}_r=\frac{1}{\gamma},
\end{equation}
is a straightforward consequence of the definition of $V_r$. 
Concerning~\eqref{dini9}, remark that for all $x\in B_X(0,r)$ and $h\geq 0$,  we have 
\begin{align*}
\sup_{\s\in \S, t\in[\bar t,\bar t+h]}\|e^{\gamma t}\phi(t,x,\s)\|&\leq \sup_{t\in[\bar t,\bar t+h]}Me^{(\gamma-\lambda) t}\|x\|
\leq Me^{(\gamma-\lambda) \bar t}\|x\|= \|x\|,
\end{align*}
and hence
\begin{equation*}
\sup_{\s\in \S, t\in[0,\bar t+h]}\|e^{\gamma t}\phi(t,x,\s)\|=
%\gamma V_r(x).
\sup_{\s\in \S, t\in[0,\bar t]}\|e^{\gamma t}\phi(t,x,\s)\|. %\|\geq \|x\|.
\end{equation*}
Then, for every $\tilde\s\in\S$, $x\in B_X(0,r)$, and $h\ge 0$, 
we have
\begin{align*}
 V_r(\phi(h,x,\tilde\s))&=\frac1\gamma\sup_{\s\in \S, t\in[0,\bar t]}\|e^{\gamma t}\phi(t,\phi(h,x,\tilde\s),\s)\|
\leq \frac1\gamma\sup_{\s\in \S, t\in[0,\bar t]}\|e^{\gamma t}\phi(t+h,x,\s)\|\nonumber\\
&= \frac{e^{-\gamma h}}{\gamma}\sup_{\s\in \S, t\in[0,\bar t]}\|e^{\gamma (t+h)}\phi(t+h,x,\s)\|
=\frac{e^{-\gamma h}}{\gamma}\sup_{\s\in \S, t\in[h,\bar t+h]}\|e^{\gamma t}\phi(t,x,\s)\|\nonumber\\
&\leq \frac{e^{-\gamma h}}{\gamma}\sup_{\s\in \S, t\in[0,\bar t+h]}\|e^{\gamma t}\phi(t,x,\s)\|=
%\leq e^{-\gamma h}\sup_{\s\in \S, t\in[0,\bar t]}\|e^{\gamma t}\phi(t,x,\s)\|\\
%&\leq e^{-\gamma h}
 e^{-\gamma h} V_r(x). 
\end{align*}
Therefore, for all $x\in B_X(0,r)$ and any $\tilde\s\in\S$, it follows that 
\begin{align*}
\overline D_{\tilde\s}V_r(x)&=\limsup_{h\downarrow 0}\frac{V_r(\phi(h,x,\tilde\s))-V_r(x)}{h}
\leq \limsup_{h\downarrow 0}\frac{e^{-\gamma h}-1}{h}V_r(x)
\leq \limsup_{h\downarrow 0}\frac{e^{-\gamma h}-1}{h}\frac{\|x\|}{\gamma}\leq -\|x\|,
\end{align*}
which implies that inequality~\eqref{dini9} holds true. 

Let us prove that the functional $V_r:X\to\R_+$ is continuous. For 
$x,y\in X$, we have
\begin{align}
 |V_r(x)-V_r(y) |
= &\left|\sup_{\s\in \S, t\in [0,\bar t]}\|e^{\gamma t}\phi(t,x,\s)\|-\sup_{\s\in \S, t\in [0,\bar t]}\|e^{\gamma t}\phi(t,y,\s)\|\right|\nonumber\\
\leq&\left| \sup_{\s\in\S, t\in [0,\bar t]}\left(\|e^{\gamma t}\phi(t,x,\s)\|-\|e^{\gamma t}\phi(t,y,\s)\|\right)\right|\nonumber\\
\leq& \sup_{\s\in\S, t\in [0,\bar t]}\left|\|e^{\gamma t}\phi(t,x,\s)\|-\|e^{\gamma t}\phi(t,y,\s)\|\right|\nonumber\\
\leq& e^{\gamma \bar t}\sup_{\s\in\S, t\in[0,\bar t]}\left| \|\phi(t,x,\s)\|-\|\phi(t,y,\s)\|\right|.\label{V-continuity}
\end{align}
The continuity of $V$ then follows from 
%By using 
the $\S$-uniform continuity of $\phi$.
%, one can choose $\eta$ small enough such that, for $y\in B_X(x,\eta)$, we have 
%\begin{eqnarray*}
%\left |V_r(x)-V_r(y)\right | < \varepsilon.
%\end{eqnarray*}
%This implies the continuity of $V_r$. 
Moreover, if the transition map $\phi$ is $\S$-uniformly Lipschitz continuous,
then we deduce from~\eqref{V-continuity}  that
\begin{eqnarray*}
\left |V_r(x)-V_r(y)\right |&\leq&e^{\gamma \bar t}l(\bar t)\|x-y\|,\label{V-Lipshitz}
\end{eqnarray*}
where $l(\bar t)$ is as in Definition~\ref{def:philip},
%the uniform Lipschitz constant associated with the transition map $\phi$, 
which implies 
the Lipschitz continuity of $V_r$. In the case where the transition map is Lipschitz continuous on bounded sets, 
we conclude similarly.

\begin{remark}
The construction of the Lyapunov functional~\eqref{lyapunov3} %, given in the proof of Theorem~\ref{coercive1}, 
is based on the classical construction given in~\cite{Pazy} in the context of linear $C_0$-semigroups. This is also used in~\cite{7402277} in order to construct  a coercive input-to-state Lyapunov functional for bilinear infinite-dimensional systems with bounded input operators. An alternative construction of a coercive common Lyapunov functional can be given by %the following 
 \begin{equation}\label{alternative1}
 V(x)=\sup_{\s\in \S}\int_{0}^{+\infty}\|\phi(t,x,\s)\|dt+\sup_{t\geq 0,\s\in\S}\|\phi(t,x,\s)\|, 
 \end{equation}
 or %even 
 also
 by 
  \begin{equation}\label{alternative2}
 V(x)=\int_{0}^{+\infty}\sup_{\s\in \S}\|\phi(t,x,\s)\|dt+\sup_{t\geq 0,\s\in\S}\|\phi(t,x,\s)\|.
 \end{equation}
For both constructions $V$ satisfies~\eqref{norm comparison2}: the right-hand inequality
% in~\eqref{norm comparison2} 
follows directly from the uniform exponential stability assumption with ${\overline c=M\left(1+1/\lambda\right)}$,
and the left-hand one
%inequality in~\eqref{norm comparison2}
holds with $\underline c=1$. Indeed, 
the second term appearing in~\eqref{alternative1} (respectively, \eqref{alternative2}) guarantees the coercivity of the functional $V$. The first term appearing in~\eqref{alternative1} (respectively, \eqref{alternative2}) is actually a (possibly non-coercive) Lyapunov functional which can be used to give a converse to Theorem~\ref{ULES-SG}. 
%Under the assumption that there exists an exponentially decreasing dynamics 
%commuting with all the switch dynamics, use~\cite{Hante2011}.
\end{remark}

\subsection{Proof of Theorem~\ref{sampling theorem}}
Fix $r>0$ and let 
$V_r:X\to \mathbb{R}_+$ be %the 
a
$L_r$-Lipschitz continuous Lyapunov functional satisfying the conclusions of 
%provided by 
Lemma~\ref{coercive3}.  
Since $V_r$ is decreasing along the trajectories of $\Sigma_0$ then, setting $\rho:=r \underline{c}_r/\overline{c}_r$, we have $\phi^{\Sigma_0}(t,x,\sigma)\in B_X(0,r)$ for every $x\in B_X(0,\rho)$, $t\ge 0$, and $\sigma\in {\cal S}={\rm PC}$.
%$\phi^{\Sigma_0}(t,x,\sigma)\in B_X(0,r)$ for every $x\in B_X(0,\rho)$, $t\ge 0$, and $\sigma\in {\cal S}={\rm PC}$.
%Define $\overline{v}=\frac12\min \{V_r(x)\mid \|x\|=r\}$. 
%Let $\rho =\frac{\overline{v}}{\overline{c}_r}>0$, so that 
%\[ V_r(x)\le \overline{v},\quad \forall~x\in B_X(0,\rho).\]
%In particular, since $V_r$ is decreasing along the trajectories of \paolo{$\Sigma_0$},
%$\phi^{\Sigma_0}(t,x,\sigma)\in B_X(0,r)$ for every $x\in B_X(0,\rho)$, $t\ge 0$, and $\sigma\in {\cal S}={\rm PC}$. 
%\footnote{\ih{I don't understand here, why we need to precise that $\phi^{\Sigma_0}(t,x,\sigma)\in B_X(0,r)$ for $x_0\in B_X(0,\rho)$?} \mario{Because we want to apply the decay of $V_r$ in the direction of the flow of $\Sigma_0$ at every point of the trajectory $x^{\Sigma}(\cdot)$}\ihab{For me we just need to precise that, starting 
%in a small enough ball, $x^{\Sigma}(t)$ stays  in $B_X(0,r)$ until some times $t^*$. Maybe there is something which is not clear to me... }\mario{At some moment we want to say something about the entire trajectory $x^{\Sigma}(\cdot)$, and this is possible because it stays in $B_X(0,r)$ for all times, which, in turns, is true because 
%$x^{\Sigma_0}(\cdot)$ stays in $B_X(0,r)$ for all times.}
%}

%
%
%
%Fix $r>0$ and let 
%$\rho:=r M(r)$ so that $\phi^{\Sigma_0}(t,x,\sigma)\in B_X(0,\rho)$ for every $x\in B_X(0,r)$, $t\ge 0$, and $\sigma\in {\cal S}={\rm PC}$. 
%Let 
%$V_\rho:X\to \mathbb{R}_+$ be the $L_\rho$-Lipschitz continuous Lyapunov functional provided by Lemma~\ref{coercive3}.  
Let $x_0\in B_X(0,\rho)$ and  $\s\in {\cal S}$. 
Let $t^*>0$ be such that the trajectory $x^\Sigma(\cdot)=\phi^{\Sigma}(\cdot,x_0,\sigma)$ of 
system $\Sigma$ stays in $B_X(0,r)$ for $t\in [0,t^*]$. 
Computing the upper Dini derivative of $V_r$ along 
$x^\Sigma(\cdot)$
 gives
\begin{align}
\overline{D}_{\s}V_r(x^{\Sigma}(t))&=\limsup_{h\downarrow 0}\frac{V_r\left(x^{\Sigma}(t+h)\right)-V_r\left(x^{\Sigma}(t)\right)}{h}\nonumber\\
                            &=\limsup_{h\downarrow 0}\Big[\frac{V_r\left(x^{\Sigma}(t+h)\right) 
                            -V_r\left(x^{\Sigma_0}(x^{\Sigma}(t),h)\right)}{h}
                            +\frac{V_r\left(x^{\Sigma_0}(x^{\Sigma}(t),h)\right)-V_r\left(x^{\Sigma}(t)\right)}{h}\Big]\nonumber\\
                            &\leq \limsup_{h\downarrow 0}\frac{V_r\left(x^{\Sigma}(t+h)\right)-V_r\left(x^{\Sigma_0}(x^{\Sigma}(t),h)\right)}{h}
                            +
                            \limsup_{h\downarrow 0}\frac{V_r\left(x^{\Sigma_0}(x^{\Sigma}(t),h)\right)-V_r\left(x^{\Sigma}(t)\right)}{h}\nonumber\\
                             &\leq L_r\limsup_{h\downarrow 0}\frac{\|x^{\Sigma}(t+h)-x^{\Sigma_0}(x^{\Sigma}(t),h)\|}{h}-\|x^{\Sigma}(t)\|.\label{derivative-sample}
                                \end{align}

Observe that 
for $k\geq 0$, $t\in [s_k,s_{k+1})\cap [0,t^*]$, and $h>0$ small enough we have 
\begin{align*}
x^{\Sigma}&(t+h)=T_{h}x^{\Sigma}(t)+
\ds\int_{0}^{h}{T_{h-s}f_{\s(t+s)}\left(x^{\Sigma}(t+s),K(T_{t+s-s_k}x^{\Sigma}(s_k))\right)}ds,
\end{align*}                               
and 
\begin{align*}
x^{\Sigma_0}&(x^{\Sigma}(t),h)=T_{h}x^{\Sigma}(t)+\ds\int_{0}^{h}{T_{h-s}f_{\s(t+s)}\left(x^{\Sigma_0}(x^{\Sigma}(t),s),K(x^{\Sigma_0}(x^{\Sigma}(t),s))\right)}ds.
\end{align*}   
Using the fact that $f$ and $K$ are, respectively, $L_f$- and $L_K$-Lipschitz continuous, 
one gets
%the following inequalities hold
\begin{align*}
\limsup_{h\downarrow 0}&\frac{\|x^{\Sigma}(t+h)-x^{\Sigma_0}(x^{\Sigma}(t),h)\|}{h}
%&\leq& M\|f_{\s_k}(x^{\Sigma}(t),K(T(t-t_k)x^{\Sigma}(t_k)))-f_{\s_k}(x^{\Sigma}(t),K(x^{\Sigma}(t)))\|\\
\leq \Gamma L_fL_K\|x^{\Sigma}(t)-T_{t-s_k}x^{\Sigma}(s_k)\|.
\end{align*}
Going back %Returning 
to~\eqref{derivative-sample}, we obtain that 
                              \begin{align}
\overline{D}_{\s}V_r&(x^{\Sigma}(t))
\leq \Gamma L_r L_fL_K\|x^{\Sigma}(t)-T_{t-s_k}x^{\Sigma}(s_k)\|-\|x^{\Sigma}(t)\|.\label{dini-sampling}
                             \end{align}
We now need to estimate $$\varepsilon (t-s_k):=\|x^{\Sigma}(t)-T_{t-s_k}x^{\Sigma}(s_k)\|, \quad t\in [s_k, s_{k+1}).$$      
By adding and subtracting  
\[\int_{0}^{t-s_k}T_{t-s_k-s}f_{\s(s+s_k)}\left(T_{s}x^{\Sigma}(s_k),K\left(T_{s}x^{\Sigma}(s_k)\right)\right)ds\] in~\eqref{sample-mild}, and using the identity $f_q(0,0)=0$, 
we obtain
\begin{align*}
\varepsilon(t-s_k)%&=&\|\int_{0}^{t-t_k}T(t-t_k-s)f\left(x^{\Sigma}(s+t_k),K\left(T(s)x^{\Sigma}(t_k)\right)\right)ds\|\\
\leq&\Gamma e^{\omega\delta}L_{f}\int_{0}^{t-s_k}\varepsilon(s)ds+\Gamma e^{\omega\delta}\int_{0}^{t-s_k}\|f_{\s(s+s_k)}\left(T_{s}x^{\Sigma}(s_k),K\left(T_{s}x^{\Sigma}(s_k)\right)\right)\|ds\nonumber\\
\leq&\Gamma e^{\omega\delta}L_{f}\int_{0}^{t-s_k}\varepsilon(s)ds+\Gamma ^2e^{2\omega\delta}L_{f}(1+L_{K})\int_{0}^{t-s_k}\|x^{\Sigma}(s_k)\|ds\nonumber\\
\leq&\Gamma e^{\omega\delta}L_{f}\int_{0}^{t-s_k}\varepsilon(s)ds+\Gamma ^2e^{2\omega\delta}L_{f}(1+L_{K})\delta\|x^{\Sigma}(s_k)\|.
\end{align*}
By Gronwall's lemma, we have 
\begin{equation}\label{estim1}
\varepsilon(t-s_k)\leq c\delta\|x^{\Sigma}(s_k)\|,
\end{equation}
where $c=\Gamma ^2e^{2\omega\delta}L_{f}(1+L_{K})e^{L_{f}\delta \Gamma e^{\omega\delta}}$. 
Observe that, by~\eqref{Groupe}, 
\begin{equation}\label{eq:ya1}
\|x^{\Sigma}(s_k)\|\leq \Gamma e^{\omega\delta}\|T_{t-s_k}x^{\Sigma}(s_k)\|. 
\end{equation}
 Hence, from~\eqref{estim1},  \eqref{eq:ya1}, and the triangular inequality,  we get 
 \begin{align*} 
\varepsilon(t-s_k)\leq c\delta \Gamma e^{\omega\delta}\|T_{t-s_k}x^{\Sigma}(s_k)\|\leq c\delta \Gamma e^{\omega\delta}\Big(\varepsilon(t-s_k)+\|x^{\Sigma}(t)\|\Big),
\end{align*}
that is, for sufficiently small $\delta$, 
 \begin{equation*} 
\varepsilon(t-s_k)\leq \frac{c\delta \Gamma e^{\omega\delta}}{1-c\delta \Gamma e^{\omega\delta}}\|x^{\Sigma}(t)\|.
\end{equation*}
Let $\delta^{\star}>0$ be such that 
\begin{equation*}\label{cstar}
c^{\star}:=\Gamma L_r L_fL_K \frac{c\delta^{\star} \Gamma e^{\omega\delta^{\star}}}{1-c\delta^{\star} \Gamma e^{\omega\delta^{\star}}}<1.
\end{equation*}
 It follows from~\eqref{dini-sampling} that 
for every $\delta\in (0,\delta^{\star})$, we have 
  \begin{align*}
\overline{D}_{\s}V_r(x^{\Sigma}(t))\leq \Gamma L_r L_fL_K\varepsilon(t-t_k)-\|x^{\Sigma}(t)\|\leq (c^{\star}-1)\|x^{\Sigma}(t)\|\leq \frac{c^{\star}-1}{\overline{c}_r}V_r(x^{\Sigma}(t)).
\end{align*}
In particular, $V_r$ decreases exponentially along $x^{\Sigma}$, which implies that $t^*$ can be taken arbitrarily large
%\footnote{\ihab{Do we need to precise here that $V_r$ is coercive (in order to get that $\Sigma$ is UES)?}}
 and, since $V_r$ is coercive, we conclude that $\Sigma$ is UES in $B_{X}(0,\rho)$.

To conclude the proof, we need to verify that $\rho$ can be taken arbitrarily large (as $r\to \infty$) if \eqref{eq:limMr} holds true. In order to do so, notice that if $V_r$ is constructed as in the proof of Theorem~\ref{coercive1}, 
then $\overline{c}_r$ and $\underline{c}_r$ can be chosen as in \eqref{eq:upb} and \eqref{eq:lob}. 
In this case $\rho = r \underline{c}_r/\overline{c}_r = r/M(r)$,
%Since, moreover, $\overline{v}\ge r \underline{c}_r/2$,  it follows that
%\[\rho=\frac{\overline{v}}{\overline{c}_r}\ge \frac{r}{2M(r)},\]
concluding the proof of the theorem.

\subsection{Proof of Theorem~\ref{iss theorem}}
Applying Theorem~\ref{coercive1}, there exist a $L_V$-Lipschitz continuous functional $V: X\to \R_+$ and $\underline c,\overline c>0$ such that 
\begin{equation} \label{eq:coerv}
\underline c\|x\|\leq V(x)\leq \overline c\|x\|, \quad \forall~x\in X, 
\end{equation}
and 
\begin{equation}\label{lp2}
\limsup_{h\downarrow 0}\frac{V(\phi_0(h,x,\s))-V(x)}{h}\leq -\|x\|,
% \quad \forall~x\in X,\; \forall~\s\in\S.
\end{equation}
for every $x\in X$ and $\s\in\S$.
Let $t\geq 0$, $\s\in \S$, and $u\in L^p(U)$, $1\leq p\leq +\infty$. Set $x=\phi_u(t,0,\s)$. For $h>0$ small enough, one has 
\begin{align}
\frac{V(\phi_u(h,x,\s))-V(x)}{h}
&=\frac{V(\phi_0(h,x,\s))-V(x)}{h} +\frac{V(\phi_u(h,x,\s))-V(\phi_0(h,x,\s))}{h}\nonumber\\
& \leq \frac{V(\phi_0(h,x,\s))-V(x)}{h}+L_{V}\frac{\|\phi_u(h,x,\s)-\phi_0(h,x,\s)\|}{h}\label{lp3}.
\end{align}
Define $e(s)=\phi_u(s,x,\s)-\phi_0(s,x,\s)$, for $s\in [0,h]$. By the variation of constant formula, one has
\begin{equation*}
e(s)=\int_{0}^{s} T_{s-\tau}\Upsilon(t,\tau,x,\s)
%\left(f_{\s(\tau)}(\phi_0(\tau,x,\s)+e(\tau),u(t+\tau))-f_{\s(\tau)}(\phi_0(\tau,x,\s),0)\right)
d\tau, \quad \forall~s\in [0,h],
\end{equation*}
where 
\begin{align*}
\Upsilon(t,\tau,x,\s)=f_{\s(\tau)}(\phi_0(\tau,x,\s)+e(\tau),u(t+\tau))-f_{\s(\tau)}(\phi_0(\tau,x,\s),0).
\end{align*}
Hence, one deduces that 
\begin{equation}\label{lp5}
\|e(h)\|\leq \Gamma L_{f}e^{\omega h}\int_{0}^{h}\left(\|e(\tau)\|+\|u(t+\tau)\|\right)d\tau. 
\end{equation}
For $p=+\infty$, one deduces that 
\begin{equation*}\label{linfini}
\limsup_{h\downarrow 0} \frac{\|e(h)\|}{h}\leq \Gamma L_{f}\|u\|_{L^{\infty}(U)}.
\end{equation*}
Letting $h\downarrow 0$ in~\eqref{lp3} and using \eqref{eq:coerv} 
%the coercivity of $V$  such that 
one gets
\begin{align*}
\limsup_{h\downarrow 0}\frac{V(\phi_u(t+h,0,\s))-V(\phi_u(t,0,\s))}{h}&\leq 
-\|\phi_u(t,0,\s)\|+\Gamma L_fL_V\|u\|_{L^{\infty}(U)}\\
&\leq 
-\frac{1}{\overline c}V(\phi_u(t,0,\s))+c\|u\|_{L^{\infty}(U)},
\end{align*}
with $c=\Gamma L_fL_V$. Using~\cite[Theorem 9]{Hagood} for the function $t\mapsto e^{\frac{t}{\overline c}}V(\phi_u(t,0,\s))$, 
one 
concludes
that 
\[V(\phi_u(t,0,\s))\le c\bar c\|u\|_{L^{\infty}(U)},\quad t\ge 0.\]
The theorem follows for $p=+\infty$.  

We next assume that $1\leq p<+\infty$.
By a standard density argument, it is enough to prove~\eqref{lp1} for those $u\in L^p(U)$ which are, in addition, continuous, since $u\mapsto \phi_u(t,0,\s)$ is continuous in $L^p(U)$.
For $u$ continuous, we deduce, from~\eqref{lp5}, that 
\begin{equation}\label{lp4}
\limsup_{h\downarrow 0} \frac{\|e(h)\|}{h}\leq L_{f}\|u(t)\|.
\end{equation}
By using~\eqref{lp2}, \eqref{lp3}, and~\eqref{lp4} one gets 
\begin{align}
\limsup_{h\downarrow 0}&\frac{V(\phi_u(t+h,0,\s))-V(\phi_u(t,0,\s))}{h}\leq -\|\phi_u(t,0,\s)\|+c\|u(t)\|,\label{eq:pi}
\end{align}
with $c=L_f L_V$.
 Let 
$\varphi_p(r)=\frac{r^{p}}{p}$, $r\ge 0$, and  
 $W_p(x)=\varphi_p(x)$, $x\in X$. 
 Since $\varphi_p$ is increasing and differentiable, one deduce from \eqref{eq:pi} that 
 \begin{align*}
\limsup_{h\downarrow 0}&\frac{W_p(\phi_u(t+h,0,\s))-W_p(\phi_u(t,0,\s))}{h}
\leq -V^{p-1}(\phi_u(t,0,\s))\|\phi_u(t,0,\s)\|
+cV^{p-1}(\phi_u(t,0,\s))\|u(t)\|.
\end{align*}
 According to~\cite[Theorem 9]{Hagood}, it follows that
\begin{align*}
0\leq W_p(\phi_u(T,0,\s))
\leq -\int_0^{T}V^{p-1}(\phi_u(t,0,\s))\|\phi_{u}(t,0,\s)\|dt
+c\int_0^{T}V^{p-1}(\phi_u(t,0,\s))\|u(t)\|dt,
 \quad \forall~T>0. 
\end{align*}
By using \eqref{eq:coerv} one gets that 
\begin{align*}
{\underline c}^{p-1}\int_0^{T}&\|\phi_u(t,0,\s)\|^pdt
\leq c{\overline c}^{p-1}\int_0^{T}\|\phi_u(t,0,\s)\|^{p-1}\|u(t)\|dt, \quad \forall~T>0. 
\end{align*}
The theorem is proved for $p=1$, by letting $T\to+\infty$. 
For $p>1$, one first applies H\"older's inequality and then lets $T\to+\infty$ to get the conclusion.

%\textcolor{red}{\begin{itemize}
%\item kdv without saturation? 
%\item la discussion Mironchenko-Wirth (comparaison des conditions RFC... avec (12)): a mettre dans l'intro. 
%\item l'exemple sur le wave si on peut le relaxer pour le cas semiglobal, en sachant $c=e^1$ dans ce cas.
%Il faut alors avoir le converse sur le semiglobal.  Il faut aussi voir si notre sampling s'\'etend \`a ce cas (les $V$ sont localisées)  
%\item Reorganiser le papier de tel sorte que le USGES devient un résultat secondaire
%\item converse avec ULES (mettre une remarque)
%\item Regular Lyapunov function (convolution est ce que c'est faisable?)
%\end{itemize}}

%\ih{
%\begin{itemize}
%\item semiglobal implies a weak form of ISS?
%\item we can use our ISS theorem to reduce the sampling theorem in the case of UGES.
%\item corollary on robustness of our sampling theorem in the case of perturbed output measurments.
%\end{itemize}}

\bibliography{biblio}

\end{document}